\documentclass[11pt,reqno]{amsart}
\usepackage[tt=false]{libertine}
\usepackage{amssymb}
\usepackage{mathtools}
\usepackage[varbb]{newpxmath}
\usepackage[foot]{amsaddr}

\let\savedbigtimes\bigtimes
\let\bigtimes\relax
\usepackage{mathabx}
\let\bigtimes\savedbigtimes

\usepackage[margin=1in, bottom=1in]{geometry}
\usepackage{graphicx}
\usepackage{caption}
\usepackage{subcaption}
\usepackage{enumitem}
\usepackage{algorithm}
\usepackage{algpseudocode}
\usepackage{comment}
\usepackage{bbm}

\usepackage[dvipsnames]{xcolor}

\usepackage[colorlinks=true,
  linkcolor=teal!60!black,
  citecolor=PineGreen,
  urlcolor=RedViolet]{hyperref}
\hypersetup{bookmarksopen=true}
\usepackage[yyyymmdd,hhmmss]{datetime}

\usepackage{tikz}
\usetikzlibrary{decorations.pathreplacing,calligraphy}
\usetikzlibrary{calc}

\newtheorem{thm}{Theorem}[section]
\newtheorem{lem}[thm]{Lemma}

\newtheorem{prop}[thm]{Proposition}

\theoremstyle{definition}

\theoremstyle{remark}

\numberwithin{equation}{section}
\def\beq#1\eeq{%
    \begin{equation}%
    #1%
    \end{equation}%
}
\def\baln#1\ealn{%
    \begin{align}%
    #1%
    \end{align}%
}
\def\balnn#1\ealnn{%
    \begin{align*}%
    #1%
    \end{align*}%
}

\def\bA{{\boldsymbol{A}}}
\def\bB{{\boldsymbol{B}}}
\def\bC{{\boldsymbol{C}}}

\def\bG{{\boldsymbol{G}}}

\def\bI{{\boldsymbol{I}}}
\def\bJ{{\boldsymbol{J}}}

\def\bW{{\boldsymbol{W}}}

\def\bY{{\boldsymbol{Y}}}
\def\bZ{{\boldsymbol{Z}}}

\def\be{{\boldsymbol{e}}}

\def\bg{{\boldsymbol{g}}}
\def\bh{{\boldsymbol{h}}}

\def\bm{{\boldsymbol{m}}}

\def\bu{{\boldsymbol{u}}}
\def\bv{{\boldsymbol{v}}}
\def\bw{{\boldsymbol{w}}}

\def\by{{\boldsymbol{y}}}
\def\bz{{\boldsymbol{z}}}
\def\bone{{\boldsymbol{1}}}

\def\bbE{{\mathbb{E}}}

\def\bbP{{\mathbb{P}}}

\def\bbR{{\mathbb{R}}}

\def\bbZ{{\mathbb{Z}}}

\def\cA{{\mathcal{A}}}

\def\cF{{\mathcal{F}}}

\def\cH{{\mathcal{H}}}

\def\cK{{\mathcal{K}}}

\def\cN{{\mathcal{N}}}

\def\cZ{{\mathcal{Z}}}

\def\de{{\mathsf{d}}}
\def\TV{{\mathsf{TV}}}
\def\KL{{\mathsf{KL}}}
\def\pro{{\mathsf{pro}}}

\def\vec{{\mathsf{vec}}}

\makeatletter
\newenvironment{breakablealgorithm}
{
		\begin{center}
			\refstepcounter{algorithm}
			\hrule height.8pt depth0pt \kern2pt
			\renewcommand{\caption}[2][\relax]{
				{\raggedright\textbf{\ALG@name~\thealgorithm} ##2\par}%
				\ifx\relax##1\relax 
				\addcontentsline{loa}{algorithm}{\protect\numberline{\thealgorithm}##2}%
				\else 
				\addcontentsline{loa}{algorithm}{\protect\numberline{\thealgorithm}##1}%
				\fi
				\kern2pt\hrule\kern2pt
			}
		}{
		\kern2pt\hrule\relax
	\end{center}
}
\makeatother

\title[Efficient Posterior Sampling for $\mathbb Z_2$ Synchronization]{Efficient Posterior Sampling for $\mathbb Z_2$ Synchronization}

\author{Zhangsong Li}

\address[Z.~Li]{School of Mathematical Sciences, Peking University}
\email[Z.~Li]{ramblerlzs@pku.edu.cn}

\date{\today}

\begin{document}

\begin{abstract}
    Consider the $\mathbb Z_2$ synchronization problem
    \[
        \bY = \frac{\lambda}{\sqrt n} \theta\theta^{\top} + \bZ,
    \]
    where $\theta$ is uniform on $\{-1,1\}^n$ and $\bZ$ is an independent Gaussian Wigner matrix with off-diagonal variance one. We give a polynomial-time posterior sampling algorithm for every fixed $\lambda>1$, for which the conditional output law converges to the posterior in total variation, in expectation over the observation. The construction combines sequential TAP proposals with an independence Metropolis correction. The key is to control signed overlaps after logarithmic pinning and conditional TAP approximations along a random revealing path, which give an efficiently evaluable proposal with a polynomial density-ratio bound outside a set of vanishing posterior mass. To the best of our knowledge, this is the first polynomial-time posterior sampler for $\mathbb Z_2$ synchronization with a total-variation guarantee throughout the supercritical regime. For comparison, the diffusion-based sampler of \cite{montanari2023posterior} gives normalized Wasserstein guarantees at sufficiently large fixed signal-to-noise ratio. 
\end{abstract}

\maketitle
\enlargethispage{3pt}

\setcounter{tocdepth}{1}
\tableofcontents

\section{Introduction and main results}\label{s:intro}

Spiked random matrices provide a fundamental model in high-dimensional statistics and probability; see, for example, \cite{johnstone2000distribution}. We consider the $\bbZ_2$ synchronization model
\[
    \bY = \frac{\lambda}{\sqrt n}\theta\theta^{\top}+\bZ,
\]
where $\theta$ is uniform on $\{-1,1\}^n$ and independent of the real symmetric Gaussian matrix $\bZ$. The entries $(Z_{ij})_{i<j}$ are independent $\cN(0,1)$ variables, the diagonal entries are independent $\cN(0,2)$ variables, and these two families are independent. For this prior and normalization, $\lambda=1$ is the threshold for weak recovery of the planted signal $\theta$: below the threshold no estimator has nonvanishing absolute normalized correlation with the spike, whereas above it a leading-eigenvector estimator is informative \cite{deshpande2017asymptotic, montanari2021estimation}. The associated spectral transition is the Wigner counterpart of the BBP transition \cite{baik2005phase, peche2006largest}. A substantial literature characterizes optimal estimation error, efficient algorithms, and higher-rank extensions \cite{dia2016mutual, deshpande2017asymptotic, el2018estimation, lelarge2017fundamental, miolane2017fundamental, barbier2019adaptive, montanari2021estimation}. In particular, TAP optimization and approximate message passing yield asymptotically Bayes-optimal matrix estimation for $\bbZ_2$ synchronization \cite{fan2021tap, celentano2023local}.

In the case of $\bbZ_2$ synchronization, given the matrix $\bY$, the posterior distribution of $\theta$ is described by an Ising model: a probability distribution $\mu_\bJ$ on $\{ \pm1 \}^n$ defined by
\begin{equation}{\label{eq:posterior}}
    \mu_{\bJ}(x) \propto \exp\left( H_{\bJ}(x) \right), \qquad H_\bJ(x)= \frac{1}{2} \langle x, \bJ x \rangle, \qquad \bJ = \frac{\lambda}{\sqrt{n}} \bY = \frac{\lambda^2}{n} \theta\theta^{\top} + \lambda \bW.
\end{equation}
Here $\bW = \frac{1}{\sqrt n}\bZ$ is an $n \times n$ GOE matrix, with off-diagonal variance $\frac{1}{n}$ and diagonal variance $\frac{2}{n}$. We also write $\mu_{\bY}=\mu_{\bJ}$ for this posterior. We show that efficient posterior sampling is possible throughout the supercritical regime $\lambda>1$. Our main result is the following.
\begin{thm}\label{thm:sampling}
    There is a randomized algorithm (see Algorithm~\ref{alg:posterior-sampler}) that takes $\bY$ and the known parameter $\lambda$ as input and outputs a sample $\widehat{\theta} \sim \widehat{\mu}$, with running time polynomial in $n$ for each fixed $\lambda>1$, whose conditional output law $\widehat\mu$ satisfies
    \begin{equation}\label{eq:main-tv}
        \bbE_{\bY} \| \widehat\mu(\cdot \mid \bY)-\mu_{\bY} \|_{\TV} \to 0.
    \end{equation}
    In particular, $\|\widehat\mu(\cdot\mid\bY)-\mu_{\bY} \|_{\TV}=o_{\bbP}(1)$. 
\end{thm}

Theorem~\ref{thm:sampling} concerns the full posterior, including its global sign symmetry, and requires no supplied side information. The running-time exponent and constants may depend on the fixed value of $\lambda>1$; we do not optimize them or obtain a bound uniform as $\lambda\downarrow1$. The error guarantee is asymptotic in $n$, without a running-time bound polynomial in a prescribed inverse error tolerance. 

Our work builds on the local-convexity theory for the TAP free energy around AMP-selected solutions \cite{celentano2023local, celentano2024sudakov} and is closely related to the diffusion-based posterior sampling framework of \cite{montanari2023posterior}. We also use the overlap and cavity identities of \cite{alaoui2017finite}. The main novelty are new estimates that control conditional TAP approximations along a random coordinate-revealing process, beginning after only logarithmically many revealed coordinates. They provide a quantitative pinning estimate and transfer local convexity from finitely many checkpoints to the intervening conditional models. Together, these estimates give a sequential proposal whose target-to-proposal density ratio is polynomially bounded outside a set of vanishing posterior mass. The standard independence Metropolis correction then yields total-variation convergence.

\subsection{Proof strategy}\label{ss:proof-overview}

We sketch the proposal construction and the main probabilistic ingredients in its analysis. The starting point is a simple reduction: to approximately sample from $\mu_\bJ$, it suffices to construct a \emph{proposal measure} $Q$ from which we can sample independently and whose log-density we can evaluate efficiently, such that the Radon-Nikodym derivative $\de \mu_\bJ/\de Q$ is polynomially bounded outside a set of vanishing expected posterior mass. As proved in Lemma~\ref{lem:metro-correction}, a standard independence Metropolis correction then yields a polynomial-time sampler with vanishing total-variation error. 

We construct the proposal distribution $Q$ sequentially. Fix $\lambda>1$, and let $I_1,\ldots,I_n$ be a uniformly random permutation of $\{ 1,\ldots,n \}$. The order is kept fixed throughout proposal generation, density evaluation, and Metropolis correction. Write $S_k=\{ I_1,\ldots,I_k \}$ and $R_k = \{ 1,\ldots,n \} \setminus S_k$. We sample the first $k_0=O_\lambda(\log n)$ signs independently and uniformly; this costs at most a factor $2^{k_0}=n^{O_\lambda(1)}$ in the target-to-proposal density ratio. At each subsequent stage, the task is to approximate the next-coordinate conditional law given the generated prefix $x_{S_k}$. The conditional posterior on $R_k$ retains the rank-one interaction and has a prefix-dependent external field, as displayed in \eqref{eq:conditioinal-distribution}. Its binary marginals are determined by the conditional mean $\bm^{(k)}(x_{S_k})$, or equivalently by the logit vector $\boldsymbol\ell^{(k)}(x_{S_k})=\operatorname{arctanh}(\bm^{(k)}(x_{S_k}))$. Here, as in Section~\ref{s:algorithm}, logits denote half-log-odds.

To approximate these marginals, we use AMP to construct centers for the constraint sets and minimize the TAP objective $\Phi_k(\cdot;x_{S_k})$ defined in \eqref{eq:def-Phi-k}. As proved in \cite{celentano2023local,celentano2024sudakov}, the optimizer $\widehat{\bm}^{(k)}(x_{S_k})$ serves as a good approximation of the true conditional mean. We then use $\widehat{\bm}^{(k)}(x_{S_k})$ to construct the logit predictor \eqref{eq:sampling-proposal-logits}. This TAP expression avoids applying $\operatorname{arctanh}$ to the approximate mean, which would amplify errors near $\pm1$. For the analysis, both the exact and computed quantities are evaluated at the same posterior-distributed prefix $x_{S_k}=\theta_{S_k}$; the proposal routine uses only its generated or supplied prefix. The approximation bound relies on two separate inputs: strong convexity on a constraint set containing the true conditional mean, and a quantitative bound on the TAP gradient at that mean.

The first difficulty is to obtain the required strong convexity \emph{simultaneously} for all $k_0\le k<n$. The local-convexity results of \cite{celentano2023local,celentano2024sudakov}, recorded in Lemmas~\ref{lem:unpinned-tap-input} and~\ref{lem:side-amp-input}, give high-probability bounds at fixed parameter values, but do not directly supply a simultaneous estimate over all revealing times. We instead choose finitely many deterministic checkpoints, with their number depending only on $\lambda$. The initial center comes from spectral AMP followed by natural-gradient refinement, with its sign chosen using the first $k_0$ coordinates. At later checkpoints, a fixed number of AMP iterations gives a new center, which is retained and restricted to the remaining coordinates until the next checkpoint. The crucial fact is that the full conditional mean $\bbE[\theta\mid\bJ,(I_a,\theta_{I_a})_{a\le k}]$ is a martingale. Thus, a maximal inequality controls its motion between checkpoints, while a deterministic principal-submatrix identity transfers the Hessian bounds. Once the remaining interaction block is sufficiently small, convexity holds on the entire cube. Lemma~\ref{lem:convexity-barycenter} combines these arguments to obtain, with high probability, both containment of the conditional mean and strong convexity throughout the revealing path.

The second difficulty is to control the TAP gradient residual uniformly in $k$, especially near $k_0$. In this range, the effective Gaussian side-channel strength $\lambda^2k/n$ tends to zero, and the signed-overlap concentration estimates in \cite{alaoui2017finite} are not uniform enough for our purpose. We therefore prove concentration of the signed posterior overlap after revealing only logarithmically many coordinates. The following consequence of Lemma~\ref{lem:uniform-pinned-overlap} states the estimate directly for the original posterior, and may be of independent interest.

\begin{prop}[Overlap concentration after logarithmic pinning] \label{prop:logarithmic-pinning}
Fix $\lambda>1$. There are constants $A_\lambda,C_\lambda<\infty$ such that the following holds for every $n\ge2$ and every integer $A_\lambda\log n\le k<n$. Let $I=(I_1,\ldots,I_n)$ be a uniformly random permutation, independent of $(\theta,\bJ)$, and put
\[
    S_k=\{I_1,\ldots,I_k\},\qquad R_k=[n]\setminus S_k,\qquad N=n-k.
\]
Conditional on $(\bJ,I,\theta_{S_k})$, let $\sigma^1,\sigma^2$ be independent samples from $\mu_\bJ(\cdot\mid x_{S_k}=\theta_{S_k})$, and write $\langle\cdot\rangle$ for their conditional average. Let $q$ be the unique positive solution of
\[
    q=\bbE\tanh^2\left( \frac{\lambda^2(Nq+k)}n +\lambda\sqrt{\frac{Nq+k}n}\,G\right), \qquad G\sim\cN(0,1).
\]
Then
\[
    \bbE\left\langle \left(\frac1N\sum_{i\in R_k}\sigma_i^1\sigma_i^2-q\right)^2 \right\rangle\le\frac{C_\lambda}{N}.
\]
Moreover, if $\bm$ and $\bC$ are the mean and covariance of the conditional posterior on $R_k$, then
\[
    \bbE\left[\operatorname{Tr}(\bC^2) +(\|\bm\|^2-Nq)^2\right]\le C_\lambda N.
\]
The expectations include the observation, planted signal, and random order, and the constants are independent of $n$ and $k$.
\end{prop}

Here $q=q_{k/n}$ in the notation of \eqref{eq:sampling-scalar-parameters}, and $A_\lambda$ may be chosen as in \eqref{eq:sampling-prefix-length}. To prove Proposition~\ref{prop:logarithmic-pinning}, we start from the absolute-overlap concentration of \cite{alaoui2017finite} and compare the pinned and unpinned posterior laws exactly. Pairs in the negative-overlap phase are unlikely to agree on every random pin. A truncation of the reciprocal prefix probability makes this comparison quantitative even after its change-of-measure cost is included. The resulting tail estimate, combined with the cavity input in Lemma~\ref{lem:cavity-moment-input}, gives the second-moment bound above. We refer to Section~\ref{ss:uniform-overlap} for more details.

The corresponding overlap estimates for the one-site cavity models yield the TAP gradient bound in Lemma~\ref{lem:tap-gradient}. Together with strong convexity and optimization accuracy, this proves Lemma~\ref{lem:construct-scalar}: there are prefix-measurable events $G_k$ with $\bbP(\cup_{k=k_0}^{n-1}G_k^c)=o(1)$ such that
\[
    \bbE\left[\bone_{G_k}
       \big\|\widehat{\boldsymbol\ell}^{(k)}(\theta_{S_k})
                 -\boldsymbol\ell^{(k)}(\theta_{S_k})\big\|^2\right]
       \le C_\lambda\frac{n-k}{n}.
\]
The logit vector is computed before inspecting $I_{k+1}$. Since this index is conditionally uniform in $R_k$, the expected squared error in the next-coordinate logit, restricted to $G_k$, is $O_\lambda(n^{-1})$. The corresponding expected conditional relative entropies, restricted to $G_k$, sum to $O_\lambda(1)$. A stopping argument at the first bad prefix, used only in the analysis, then gives the density-ratio bound at threshold $2^{k_0}n$ outside a set of vanishing expected posterior mass; see Lemma~\ref{lem:nu-is-warm}. This is the proposal guarantee required by the Metropolis reduction.


\subsection{Related sampling results}

For posterior sampling, the existing guarantees have different parameter and metric requirements. The discrete-prior result of \cite[Theorem~3]{montanari2023posterior} samples in normalized $W_2$ distance at sufficiently large fixed signal-to-noise ratio $\lambda$. With an additional Gaussian side channel of sufficiently large strength, \cite[Corollary~4.1]{bandeira2026mixing} proves polynomial-time Glauber mixing in total variation for every fixed $\lambda>0$. Related Langevin mixing results in \cite{gheissari2026mixing} also provide sampling guarantees for sufficiently large $\lambda$, although they concern a spherical spike rather than the Rademacher posterior as here.

A related model is the Sherrington--Kirkpatrick model with an external field in the sign direction $\theta$:
\[
    \widetilde{\mu}(x) \propto \exp\left( \widetilde{H}_{\bJ}(x) \right), \qquad \widetilde{H}_{\bJ}(x) = \frac{\beta}{2} \langle x, \bW x \rangle + \kappa \langle x,\theta \rangle, \qquad \beta\ge0.
\]
For the zero-field model ($\kappa=0$), the spectral criterion of \cite{eldan2022spectral} gives polynomial-time Glauber mixing for $\beta<1/4$, and \cite{anari2024trickle} extends this range to $\beta<0.295$. More recently, \cite[Corollary~1.6]{boban2026rank} proves polynomial mixing for $\beta<1/2+\varepsilon$, for an absolute constant $\varepsilon>0$. Algorithmic stochastic localization provides a different route: normalized Wasserstein sampling guarantees hold throughout $\beta<1$ \cite{el2022sampling,celentano2024sudakov}, and \cite[Theorem~1.1]{lee2026potential} proves polynomial-time sampling with vanishing total-variation error in this entire high temperature regime. Earlier progress toward total-variation sampling includes \cite{davies2026potential,davies2026weak}.

The low-temperature regime has a different structure. At zero field, the stable-algorithm obstruction of \cite{el2022sampling} applies for every $\beta>1$, while \cite{sellke2025exponentially} proves exponentially slow worst-case Glauber mixing for sufficiently large $\beta$. With a nonzero external field, replica symmetry holds in the region specified by the Almeida--Thouless condition \cite{dealmeida1978stability,lopatto2026replica}. For the displayed sign field, the change of variables $x_i\mapsto\theta_i x_i$ and the invariance of GOE under diagonal sign conjugation give the uniform-field model in distribution; the sign of $\kappa$ can also be reversed. Bolthausen's iterative TAP scheme converges in the asymptotic mean-square sense in this region, with geometric decay of successive-iterate error under the strict Almeida--Thouless condition \cite[Theorem~2.1]{bolthausen2014iterative}. These conclusions take the system-size limit before the iteration limit and do not by themselves give a mixing bound. For every fixed $\beta$, polynomial-time Glauber mixing is nevertheless available at sufficiently large field strength by \cite[Corollary~1.2]{bandeira2026mixing}. These results motivate the study of sampling in structured low-temperature models. The posterior \eqref{eq:posterior} has the Bayes-imposed relation between its rank-one interaction and Gaussian disorder. Our proof uses this relation throughout; it does not establish rapid mixing of Glauber dynamics or sampling for general SK models throughout the Almeida--Thouless region.

\section{Overview of the algorithm}\label{s:algorithm}

We first explain how a sequential proposal and an independence Metropolis correction fit together. The proposal construction is specified below, the probabilistic estimates are proved in Section~\ref{s:auxiliary-proofs}, and the complete sampler appears in Section~\ref{s:completion}.

Throughout, vector norms are Euclidean, $\|\cdot\|_{\mathrm{op}}$ and $\|\cdot\|_F$ denote the operator and Frobenius norms, and $B(\bv,R)$ is the closed Euclidean ball. The notation $\bJ_{R,R}$ retains the original coordinate labels. Probabilities and expectations in the analysis are under the joint planted law, including an independent revealing order, unless conditioning is specified. By Lemma~\ref{lem:nishimori-input}, this is equivalent to first sampling the observation and then sampling $\theta$ from its posterior. We access the real observation and known parameter through deterministic dyadic approximations. Computational cost includes all requested bits and subsequent bit operations. Polynomial-time density evaluation means computing the log-density to a prescribed inverse-polynomial absolute accuracy, with the deterministic maps defining the proposal kept fixed when score accuracy is increased. Exact random draws and scores specify the ideal output law. Appendix~\ref{s:numerical-implementation} gives the capped implementation and bounds the additional error from finite-bit draws and scores.

\subsection{Metropolis correction}
The starting point of our algorithm is the following observation, which reduces the sampling task to generating a ``warm'' proposal whose target-to-proposal density ratio is bounded outside a set of small target mass. For two strictly positive probability measures $P,Q$ on a finite set, the independence Metropolis process is the Markov chain $(X_t)_{t\ge0}$ initialized with $X_0\sim Q$. Given $X_t=y$, independently propose $z\sim Q$ and accept the proposal with probability
\begin{equation}\label{eq:metro-transition}
    \min\left\{1,\frac{P(z)Q(y)}{P(y)Q(z)}\right\}.
\end{equation}
On acceptance set $X_{t+1}=z$; otherwise set $X_{t+1}=y$.
\begin{lem}{\label{lem:metro-correction}}
    For $M\ge1$, we say $Q$ is $(\delta,M)$-warm with respect to $P$ if $P(\frac{\de P}{\de Q}>M)=\delta$. Let $K$ be the transition kernel of the process $X_t$. If $Q$ is $(\delta,M)$-warm with respect to $P$, then for every integer $t\ge0$,
    \begin{equation}\label{eq:imh-tail}
        \left\| QK^t-P \right\|_{\TV} \le \delta+\frac{\delta}{M} +\exp\left(-\frac{t(1-\delta)}{M}\right).
    \end{equation}
\end{lem}
\begin{proof}
Write $w=P/Q$ and $a(x)=\sum_zQ(z)\min\{1,w(z)/w(x)\}$ for the acceptance probability, including proposals equal to the current state. The independence-sampler bound of \cite[Theorem~6]{roberts2011quantitative}, averaged over the initial state, gives
\[
    \|QK^t-P\|_{\TV}
    \le \bbE_{X\sim Q,\,Z\sim P}
       [\{1-\min(a(X),a(Z))\}^t],
\]
where $X$ and $Z$ are independent. The bound applies to finite state spaces as well; acceptance of the same proposal still couples the two chains. When $w(x)\le M$, we have $a(x)\ge (1-\delta)/M$. Also $Q(w>M)\le\delta/M$. Restricting the expectation to $\{w(X),w(Z)\le M\}$ and using $(1-u)^t\le e^{-tu}$ proves \eqref{eq:imh-tail}.
\end{proof}
By Lemma~\ref{lem:metro-correction}, it suffices to construct a polynomial-time proposal oracle $\cA^{\pro}$ which, given $\bJ$ and $\lambda$, generates independent samples $\widehat\theta^{\pro}\sim Q$ and computes $\log Q(x)$ to any prescribed inverse-polynomial absolute accuracy at any prescribed $x\in\{-1,1\}^n$. Suppose that, for a known deterministic $M=M_n\ge1$ with $M=n^{O_\lambda(1)}$, the proposal is $(\delta,M)$-warm with respect to $\mu_\bJ$, where $\bbE\delta=o(1)$. Here the expectation includes the data and any auxiliary randomness used to construct $Q$; this auxiliary randomness, and hence the proposal distribution, is fixed throughout the correction. Initialize $X_0\sim Q$. At each step, given $X_t=y$, independently generate $z\sim Q$ using the oracle and accept it with probability
\[
    \min\left\{ 1,\exp\{H_\bJ(z)-H_\bJ(y)\} \frac{Q(y)}{Q(z)} \right\};
\]
on rejection, retain $y$. Run $T=\lceil Mn\rceil$ steps and return $\widehat\theta=X_T$. The correction is summarized in Algorithm~\ref{alg:metropolis-correction}.

\medskip
\begin{breakablealgorithm}
\caption{Independence Metropolis correction from a warm proposal}
\label{alg:metropolis-correction}
\begin{algorithmic}[1]
\Require Interaction matrix $\bJ$, a known warmness parameter $M\ge1$, and a proposal oracle $\cA^{\pro}$. 
\Ensure A configuration $\widehat\theta$ obtained after $T=\lceil Mn\rceil$ correction steps.
\State Keep all auxiliary randomness defining $Q$ fixed throughout the run.
\State Set $T\gets\lceil Mn\rceil$ and draw $X_0\sim Q$ using $\cA^{\pro}$.
\For{$t=0,\ldots,T-1$}
    \State Draw $z\sim Q$ using $\cA^{\pro}$, independently of the past.
    \State Evaluate $Q(X_t)$ and $Q(z)$ using the same fixed proposal.
    \State $\Delta\gets H_\bJ(z)-H_\bJ(X_t)+\log Q(X_t)-\log Q(z)$.
    \State Draw $U\sim\operatorname{Unif}(0,1)$ independently.
    \If{$\log U\le\min\{0,\Delta\}$}
        \State $X_{t+1}\gets z$.
    \Else
        \State $X_{t+1}\gets X_t$.
    \EndIf
\EndFor
\State \Return $\widehat\theta=X_T$.
\end{algorithmic}
\end{breakablealgorithm}

\medskip
For each fixed proposal with $\delta\le1/2$, Lemma~\ref{lem:metro-correction} gives $\|QK^T-\mu_\bJ\|_{\TV}\le2\delta+e^{-n/2}$; the same bound holds when $\delta>1/2$ by the trivial bound on total variation. Thus, writing $\widehat\mu(\cdot\mid\bJ)$ for the conditional output law and averaging over the auxiliary randomness, convexity of total variation gives
\[
    \bbE_\bJ\bigl\|\widehat\mu(\cdot\mid\bJ)-\mu_\bJ\bigr\|_{\TV} \le 2\bbE\delta+e^{-n/2}=o(1).
\]
Since each oracle call and each Hamiltonian evaluation take polynomial time, the corrected sampler also runs in polynomial time. Hence it is enough to construct an efficiently samplable and evaluable proposal for which $\de\mu_\bJ/\de Q$ is polynomially bounded outside a set of vanishing posterior mass.

\subsection{Sequential construction of the proposal}

We now describe how to generate the warm distribution $Q$. Set
\begin{equation}\label{eq:sampling-prefix-length}
    A=A_\lambda=\max\left\{256, \frac{64}{\log(2/(1+\lambda^{-2}))} \right\}, \qquad k_0=\min\{n,\lceil A\log n\rceil\}.
\end{equation}
All logarithms are natural. 
Choose a uniform permutation $I=(I_1,\ldots,I_n)$ independently of the data, and keep it fixed for the entire Metropolis run. Write $S_k=\{I_1,\ldots,I_k\}$ and $R_k=[n]\setminus S_k$. Generate $x_{I_1},\ldots,x_{I_{k_0}}$ as independent uniform signs. A prefix consists of the exposed index--value pairs $(I_a,x_{I_a})_{1\le a\le k}$; the argument $x_{S_k}$ always refers to these already exposed coordinates, not the unused part of the order. To distinguish a general input from a target-distributed input, first define the exact conditional quantities for any $x_{S_k}\in\{-1,1\}^{S_k}$ and $0\le k<n$ by
\[
\begin{aligned}
    h_j^{(k)}(x_{S_k})&=\sum_{a=1}^k J_{j,I_a}x_{I_a},\qquad j\in R_k,\\
    \bm^{(k)}(x_{S_k})
       &=\bbE[\theta_{R_k}\mid\bJ,I,\theta_{S_k}=x_{S_k}].
\end{aligned}
\]
Along the target revealing path these become $\bh^{(k)}(\theta_{S_k})$ and $\bm^{(k)}(\theta_{S_k})$, with $\theta\sim\mu_\bJ$. These target coordinates are used only in the analysis, not supplied to the sampler. Expanding $H_\bJ$ and discarding the term depending only on the revealed coordinates shows that
\begin{equation}{\label{eq:conditioinal-distribution}}
    \bbP(\theta_{R_k}=u\mid\bJ,I,\theta_{S_k}=x_{S_k}) \propto \exp\left\{\frac12u^{\top}\bJ_{R_k,R_k}u +(\bh^{(k)}(x_{S_k}))^{\top}u\right\}, \qquad u\in\{-1,1\}^{R_k}.
\end{equation}
For $i=I_{k+1}$, summing over the remaining coordinates $R_{k+1}=R_k\setminus\{i\}$ gives the exact conditional half-log-odds
\[
    \ell_i^{(k)}(x_{S_k}) =h_i^{(k)}(x_{S_k})+\frac12\log \frac{\displaystyle \sum_{u\in\{-1,1\}^{R_{k+1}}} \exp\left\{\frac12u^{\top}\bJ_{R_{k+1},R_{k+1}}u +\sum_{j\in R_{k+1}}(h_j^{(k)}(x_{S_k})+J_{ij})u_j\right\}}{\displaystyle \sum_{u\in\{-1,1\}^{R_{k+1}}} \exp\left\{ \frac12 u^{\top}\bJ_{R_{k+1},R_{k+1}}u +\sum_{j\in R_{k+1}}(h_j^{(k)}(x_{S_k})-J_{ij})u_j\right\}}.
\]
The diagonal contribution $J_{ii}/2$ is independent of the sign of $\theta_i$ and cancels in this ratio; when $k=n-1$, both sums equal one. Normalizing the two sign probabilities, we obtain
\[
    \bbP(\theta_{I_{k+1}}=b\mid\bJ,I,\theta_{S_k}=x_{S_k}) =\frac{1+b m_{I_{k+1}}^{(k)}(x_{S_k})}{2} =\frac{\exp(b\ell_{I_{k+1}}^{(k)}(x_{S_k}))}{2\cosh(\ell_{I_{k+1}}^{(k)}(x_{S_k}))}, \qquad b\in\{-1,1\}.
\]
In particular, $\ell_i^{(k)}(x_{S_k})=\operatorname{arctanh}(m_i^{(k)}(x_{S_k}))$ for every $i\in R_k$. Direct evaluation of the displayed sums involves $2^{n-k-1}$ configurations, so this exact formula does not yet give an efficient sampler. Instead, using only the supplied prefix, we compute an approximation $\widehat{\boldsymbol\ell}^{(k)}(x_{S_k})$. Only then inspect $I_{k+1}$ and sample its sign from
\begin{equation}\label{eq:sampling-binary-conditional}
    q_k(b\mid x_{S_k})= \frac{\exp(b\widehat\ell^{(k)}_{I_{k+1}}(x_{S_k}))}{2\cosh(\widehat\ell^{(k)}_{I_{k+1}}(x_{S_k}))}, \qquad b\in\{-1,1\}.
\end{equation}
Precisely, writing $Q$ for this proposal, with the data, order, and record suppressed, its density is
\begin{equation}{\label{eq:def-nu}}
    Q(x)=2^{-k_0}\prod_{k=k_0}^{n-1} \frac{\exp( x_{I_{k+1}}\widehat\ell^{(k)}_{I_{k+1}}(x_{S_k}))}{2\cosh(\widehat\ell^{(k)}_{I_{k+1}}(x_{S_k}))}.
\end{equation}
\begin{lem}\label{lem:construct-scalar}
    Fix $\lambda>1$. For $n\ge2$ and every $k_0\le k<n$, there exists an algorithm $\cA^{\vec}$ that computes a vector $\widehat{\boldsymbol\ell}^{(k)}(x_{S_k})$ from $\bJ,\lambda$ and the exposed pairs $(I_a,x_{I_a})_{1\le a\le k}$ in time $n^{O_\lambda(1)}$ with the following properties. For $\theta\sim\mu_\bJ$, let $\cH_k=\sigma(\bJ,(I_a,\theta_{I_a})_{a\le k})$. There are events $G_k \in \cH_k$ and a constant $C=C_\lambda<\infty$ such that $b_n:=\bbP(\cup_{k=k_0}^{n-1} G_k^c)=o(1)$ and, uniformly for $k_0\le k<n$,
    \begin{equation}\label{eq:vector-logit-error}
    \begin{aligned}
        \bbE\left[\bone_{G_k}
          \big\|\widehat{\boldsymbol\ell}^{(k)}(\theta_{S_k})
                    -\boldsymbol\ell^{(k)}(\theta_{S_k})\big\|^2\right]
            &\le \frac{C(n-k)}n,\\
        \bbE\left[\bone_{G_k}
          \big|\widehat\ell^{(k)}_{I_{k+1}}(\theta_{S_k})
                    -\ell^{(k)}_{I_{k+1}}(\theta_{S_k})\big|^2\right]
            &\le\frac Cn.
    \end{aligned}
    \end{equation}
    In these estimates, the algorithmic vector and the exact vector are evaluated at the same posterior-distributed prefix $\theta_{S_k}$. The expectations include $\bJ$, the independent random order $I$, and $\theta\sim\mu_\bJ$; no accuracy bound is asserted for an arbitrary input prefix or for a prefix sampled from $Q$. In particular, we have the simultaneous pathwise estimate
    \[
        \bbE_{\bJ,I}\mu_\bJ\left( \sum_{k=k_0}^{n-1} \frac{\big\|\widehat{\boldsymbol\ell}^{(k)}(\theta_{S_k})-\boldsymbol\ell^{(k)}(\theta_{S_k})\big\|^2}{n-k} >\log n\right) \le b_n+\frac{C}{\log n}=o(1).
    \]
\end{lem}

The proof of Lemma~\ref{lem:construct-scalar} is incorporated in Section~\ref{ss:sampling-guarantee}. Now we can show that the proposal $Q$ generated is indeed warm with respect to $\mu_\bJ$.
\begin{lem}{\label{lem:nu-is-warm}}
    The proposal $Q$ in \eqref{eq:def-nu} is strictly positive, admits polynomial-time sampling and log-density evaluation in the computational convention above. In addition, for every $u>0$,
    \begin{equation}\label{eq:proposal-warmness-tail}
        \bbE_{\bJ,I}\mu_\bJ\left( \frac{\de\mu_\bJ}{\de Q}>2^{k_0}e^u\right) \le b_n+\frac{C+1/e}{u}.
    \end{equation}
    Consequently, it is $(\delta,M)$-warm with respect to $\mu_\bJ$, where
    \[
        M=2^{k_0}n=n^{O_\lambda(1)},\qquad \delta=\mu_\bJ \left(\frac{\de\mu_\bJ}{\de Q}>M\right), \qquad \bbE_{\bJ,I}\delta \le b_n+\frac{C+1/e}{\log n}=o(1).
    \]
\end{lem}
\begin{proof}
    Write $p_k(\cdot\mid\theta_{S_k})$ for the exact next-coordinate conditional. Its half-log-odds are $\ell_{I_{k+1}}^{(k)}(\theta_{S_k})$, while those of $q_k(\cdot\mid\theta_{S_k})$ are $\widehat\ell_{I_{k+1}}^{(k)}(\theta_{S_k})$. Thus both factors are compared at the same target prefix. For $P_\ell(b)=e^{b\ell}/(2\cosh\ell)$, Taylor's theorem and $(\log\cosh)''\le1$ give
    \[
        \KL(P_\ell\Vert P_{\widehat\ell}) =\log\cosh\widehat\ell-\log\cosh\ell -(\widehat\ell-\ell)\tanh\ell \le\tfrac12(\widehat\ell-\ell)^2.
    \]
    Consequently, \eqref{eq:vector-logit-error} implies
    \[
        \sum_{k=k_0}^{n-1}\bbE\left[ \bone_{G_k} \KL \bigl(p_k(\cdot\mid\theta_{S_k}) \Vert q_k(\cdot\mid\theta_{S_k})\bigr)\right] \le C.
    \]
    To turn it into a likelihood-ratio estimate without controlling bad prefixes, fix $(\bJ,I)$ and an initial prefix $a=\theta_{S_{k_0}}$. Define an auxiliary suffix distribution $\widetilde\nu_a$ that uses $q_k$ as long as every good event up to and including time $k$ holds, and uses the exact conditionals $p_k$ from the first failure onward. Prefix measurability makes this a well-defined positive distribution. It is only a proof device, not part of the algorithm. Under the conditional target suffix law, set
    \[
        R=\frac{\mu_\bJ(\theta_{R_{k_0}}\mid\theta_{S_{k_0}}=a)}{\widetilde\nu_a(\theta_{R_{k_0}})}.
    \]
    The chain rule for relative entropy, averaged over $(\bJ,I)$ and the target-distributed initial prefix, gives
    \[
        \bbE\log R =\sum_{k=k_0}^{n-1}\bbE\left[ \bone_{\cap_{j=k_0}^k G_j} \KL\bigl(p_k(\cdot\mid\theta_{S_k}) \Vert q_k(\cdot\mid\theta_{S_k})\bigr)\right]\le C.
    \]
    For any two positive probability mass functions $P',Q'$ and $R'=P'/Q'$, one has
    \[
        \bbE_{P'}(-\log R')_+ =\sum_{x:R'(x)<1}Q'(x)R'(x)\log\frac1{R'(x)} \le\frac1e.
    \]
    Applying this conditionally to the suffix distributions yields $\bbE(\log R)_+\le C+1/e$, and hence $\bbP(\log R>u) \le(C+1/e)/u$ for every $u>0$. On a path with no bad prefix, the actual and auxiliary suffix probabilities agree. Since the proposal assigns probability $2^{-k_0}$ to its initial prefix,
    \[
        \frac{\mu_\bJ(\theta)}{Q(\theta)} =2^{k_0}\mu_\bJ(\theta_{S_{k_0}})R \le2^{k_0}R.
    \]
    Adding the probability $b_n$ of a bad path proves \eqref{eq:proposal-warmness-tail}. Taking $u=\log n$ gives the stated $(\delta,M)$-warmness. If $k_0=n$, the suffix is empty, the proposal is uniform, and the same conclusions hold directly with $\delta=0$.
\end{proof}

\subsection{TAP optimization and the constraint sets}\label{ss:conditional-tap}

We now specify the algorithm $\cA^{\vec}$ that computes the vector $\widehat{\boldsymbol\ell}^{(k)}(x_{S_k})$ in Lemma~\ref{lem:construct-scalar}. The estimates for this construction are deferred to Section~\ref{s:auxiliary-proofs}. For $\alpha\in[0,1]$, set $r_\alpha=\lambda^2(1-\alpha)$ and $t_\alpha=\lambda^2\alpha$, and define
\begin{equation}\label{eq:sampling-scalar-parameters}
    F(u)=\bbE\tanh^2(u+\sqrt u\,G),\qquad
    q_\alpha=F(r_\alpha q_\alpha+t_\alpha),\qquad
    a_\alpha=r_\alpha(1-q_\alpha),\quad G\sim\cN(0,1).
\end{equation}
Use the positive solution for $q_\alpha$, discarding the zero solution at $\alpha=0$. These scalar parameters depend only on $\lambda$ and $\alpha$, not on the prefix values. Precompute $a_{k/n}$ for $k_0\le k<n$ by one-dimensional quadrature and bisection, using $[1-\lambda^{-2},1]$ as the bracket for $q_{k/n}$ and a fixed deterministic numerical rule. Require absolute error at most $n^{-120}$ in each coefficient $a_{k/n}$; the notation $a_{k/n}$ in the analytic estimates always denotes its exact value. 
Given a generated or supplied prefix $x_{S_k}$, form the field
\[
    \widehat h_i^{(k)}(x_{S_k})=\sum_{a=1}^kJ_{i,I_a}x_{I_a}=h_i^{(k)}(x_{S_k}),\qquad i\in R_k.
\]
The hat records that the field is formed from the algorithm's input; it is not an approximation to a conditional mean. The equality above is between the exact-arithmetic field formulas, with numerical errors handled by the prescribed precision rules. With $\phi(v)=\frac{1+v}{2}\log\frac{1+v}{2}+\frac{1-v}{2}\log\frac{1-v}{2}$, use the objective
\begin{equation}{\label{eq:def-Phi-k}}
    \Phi_k(\bu;x_{S_k})=\sum_{i\in R_k}\phi(u_i) -\frac12\bu^{\top}\bJ_{R_k,R_k}\bu -(\widehat{\bh}^{(k)}(x_{S_k}))^{\top}\bu+\frac{a_{k/n}}2\|\bu\|^2, \qquad \bu\in(-1,1)^{R_k}.
\end{equation}
Here $\bu$ is the optimization variable, whereas $x_{S_k}$ is input data specifying the objective. Throughout, $\nabla_{\bu}$ and $\nabla_{\bu}^2$ differentiate only the first argument. The algorithm will return an approximate constrained minimizer $\widehat{\bm}^{(k)}(x_{S_k})$. The exact conditional mean $\bm^{(k)}(x_{S_k})$ need not be a stationary point of $\Phi_k(\cdot;x_{S_k})$. For the target input $x_{S_k}=\theta_{S_k}$, we will control its gradient residual and use strong convexity to compare it with the constrained minimizer for that same input. The resulting proposal uses the entire free-coordinate vector
\begin{equation}\label{eq:sampling-proposal-logits}
    \widehat{\boldsymbol\ell}^{(k)}(x_{S_k})
       =\widehat{\bh}^{(k)}(x_{S_k})+
         (\bJ_{R_k,R_k}-a_{k/n}\bI)\widehat{\bm}^{(k)}(x_{S_k}).
\end{equation}
This linear expression is important: it avoids applying $\operatorname{arctanh}$ to an approximate mean, which would amplify errors near $\pm1$. The displayed formula specifies the exact-arithmetic target for the final logit calculation; its implemented value has the error prescribed in Section~\ref{ss:deterministic-tap-optimization}.

A tuning record consists of integers $M,L_{\rm init},L_0,\ldots,L_{M-1}\ge1$, a step size $\gamma\in(0,1)$, radii $\rho_0,\rho>0$, and fractions $0<\alpha_0<\cdots<\alpha_M<1$. A valid record is fixed independently of $n$; its existence is proved in Section~\ref{ss:checkpoint-construction}. We use one such record throughout Sections~\ref{s:algorithm}--\ref{s:auxiliary-proofs}. Appendix~\ref{ss:tuning-free-sampler} removes the need to supply it by forming a finite mixture of explicitly enumerated candidates. The algorithm does not need to recognize valid records. Write $j_\ell=\lfloor\alpha_\ell n\rfloor$ and put
\[
    c_n=1-2^{-n},\qquad D_{R,n}=[-c_n,c_n]^R.
\]
The objective is defined on $(-1,1)^R$, but the numerical routine uses the closed clipped cube. This ensures that projection exists and the entropy gradient is bounded. The cutoff will not exclude any true conditional mean on the operator-norm event, for fixed $\lambda$ and sufficiently large $n$.

For the fixed record, compute the initial center once, before drawing $I$, with the safeguards of Appendix~\ref{ss:tuning-free-sampler}. Choose a deterministic sign for an eigenvector associated with the largest eigenvalue of $\bJ$, and normalize it to obtain $\bz^0$ with $\|\bz^0\|^2=n\lambda^2(\lambda^2-1)$. Initialize $\bu^{-1}=\lambda\bz^0$ and perform
\begin{equation}\label{eq:sampling-initial-amp}
    \bu^\ell=\tanh\bz^\ell,\qquad
    \bz^{\ell+1}=\bJ\bu^\ell
       -\lambda^2\left(1-\frac{\|\bu^\ell\|^2}{n}\right)\bu^{\ell-1},
       \qquad 0\le\ell<L_{\rm init}.
\end{equation}
Starting from $\bz=\bz^{L_{\rm init}}$, perform $n$ natural-gradient updates
\[
    \bu=\tanh\bz,\qquad
    \bz\leftarrow(1-\gamma)\bz+
       \gamma\left\{\bJ\bu-
          \lambda^2\left(1-\frac{\|\bu\|^2}{n}\right)\bu\right\},
\]
and store $\bv=\operatorname{Proj}_{D_{[n],n}}(\tanh\bz)$. 
The stored $\bv$ is kept fixed throughout the Metropolis run. For each generated or supplied configuration, use its first $k_0$ coordinates to set
\begin{equation}\label{eq:sampling-phase-orientation}
    \bv^+(x_{S_{k_0}})=\operatorname{sign}\left(
       \sum_{a=1}^{k_0}x_{I_a}v_{I_a}\right)\bv,
       \qquad \operatorname{sign}(0)=0.
\end{equation}
Thus $\bv$ is independent of the input prefix, whereas $\bv^+(x_{S_{k_0}})$ depends only on its first $k_0$ signs. Only this orientation, not the initial computation, is repeated when a proposal is generated or evaluated.

At a checkpoint $j=j_\ell$, $0\le\ell<M$, form $\widehat{\bh}^{(j)}(x_{S_j})$ from the current prefix. With $N=n-j$ and $r=r_{j/n}$, initialize $\bz^0(x_{S_j})=\widehat{\bh}^{(j)}(x_{S_j})$, $\bu^{-1}(x_{S_j})=0$, and perform
\begin{equation}\label{eq:sampling-side-amp}
\begin{aligned}
    \bu^t(x_{S_j})&=\tanh\bz^t(x_{S_j}),\\
    \bz^{t+1}(x_{S_j})&=\bJ_{R_j,R_j}\bu^t(x_{S_j})
                         +\widehat{\bh}^{(j)}(x_{S_j}) -r\left(1-\frac{\|\bu^t(x_{S_j})\|^2}{N}\right)
                         \bu^{t-1}(x_{S_j}),\qquad 0\le t<L_\ell.
\end{aligned}
\end{equation}
Store $\bv^{(j)}(x_{S_j})=\operatorname{Proj}_{D_{R_j,n}}(\tanh\bz^{L_\ell}(x_{S_j}))$. Between checkpoints, keep $\bv^{(j)}(x_{S_j})$ fixed and restrict it to the current free coordinates; its input remains the checkpoint subprefix, not the enlarged prefix $x_{S_k}$. Define
\begin{equation}\label{eq:sampling-checkpoint-domains}
    \cK_k(x_{S_k})=
    \begin{cases}
        D_{R_k,n}\cap B(\bv^+(x_{S_{k_0}})|_{R_k},\rho_0\sqrt n),&k_0\le k<j_0,\\[2pt]
        D_{R_k,n}\cap B(\bv^{(j_\ell)}(x_{S_{j_\ell}})|_{R_k},\rho\sqrt n),
            &j_\ell\le k<j_{\ell+1},\quad 0\le\ell<M,\\[2pt]
        D_{R_k,n},&k\ge j_M.
    \end{cases}
\end{equation}
The balls have radius proportional to $\sqrt n$, not $\sqrt{n-k}$. Each stored center belongs to its clipped cube; its restriction therefore belongs to the corresponding cube--ball intersection. Consequently these domains are nonempty, compact, and convex for every prefix. Use the restricted center as the initial optimization iterate, and use zero in the whole-cube case. If a safeguard has already assigned the record or its remaining suffix the uniform proposal, set all unused domains to $D_{R_k,n}$ with center zero. 

The next lemma supplies the additional properties needed along a target-distributed revealing path. The set $\cK_k(x_{S_k})$ is convex for every input prefix. In contrast, strong convexity of the generally nonconvex objective is asserted with high probability for $\Phi_k(\cdot;\theta_{S_k})$ on $\cK_k(\theta_{S_k})$.
\begin{lem}\label{lem:convexity-barycenter}
    Fix $\lambda>1$. There is a choice of the fixed tuning record above such that the domains $\cK_k(x_{S_k})$ defined in \eqref{eq:sampling-checkpoint-domains} are nonempty, compact, and convex for every input prefix $x_{S_k}$. Moreover, there is a constant $\kappa_\lambda>0$ and, along a posterior-distributed revealing path, events $G_k\in\cH_k$ such that
    \[
        b_n:=\bbP\left(\cup_{k=k_0}^{n-1}G_k^c\right)=o(1),
    \]
    and, on $G_k$,
    \[
    \begin{gathered}
        \|\bJ\|_{\mathrm{op}}\le\lambda^2+10\lambda,
        \qquad \bm^{(k)}(\theta_{S_k})\in\cK_k(\theta_{S_k}),\\
        \nabla_{\bu}^2\Phi_k(\bu;\theta_{S_k})\succeq\kappa_\lambda\bI
                 \quad\text{for every }\bu\in\cK_k(\theta_{S_k}).
    \end{gathered}
    \]
    Both $\Phi_k(\cdot;\theta_{S_k})$ and $\cK_k(\theta_{S_k})$ are constructed with input $x_{S_k}=\theta_{S_k}$. In particular, $\widehat{\bh}^{(k)}(\theta_{S_k})=\bh^{(k)}(\theta_{S_k})$ by the field definition. The Hessian bound is uniform over the optimization variable in this set; it is neither a bound only at the conditional mean nor a simultaneous assertion for all input prefixes. The iteration counts, radii, and checkpoint fractions are independent of $n$.
\end{lem}

The proof of Lemma~\ref{lem:convexity-barycenter} is incorporated in Section~\ref{ss:checkpoint-construction}.
Based on Lemma~\ref{lem:convexity-barycenter}, set
\[
    B_n=\sqrt n\{n^2+1+2\|\bJ\|_F+\lambda^2\},\qquad
    D_n=2\sqrt n,\qquad L_n=\lceil n^{40}\rceil,\qquad
    \gamma_n=\frac{D_n}{B_n\sqrt{L_n}}.
\]
For a fixed input prefix $x_{S_k}$, run the following projected-gradient routine for exactly $L_n$ iterations and return the average of its pre-update iterates. Within this call, $\bu^t$, $\bg$, and $\bw$ are local working variables for that fixed prefix; only their local notation suppresses the argument.

\medskip
\begin{breakablealgorithm}
\caption{Deterministic constrained TAP optimization: \textsc{TAPOptimize}}
\label{alg:tap-optimization}
\begin{algorithmic}[1]
\Require $\bJ$, input prefix $x_{S_k}$, its field $\widehat{\bh}^{(k)}(x_{S_k})$, $a_{k/n}$, and $\cK_k(x_{S_k})$ with its stored center.
\Ensure The computed vector $\widehat{\bm}^{(k)}(x_{S_k})$, or a computation-failure flag.
\State Set $\bu^0$ to the restricted stored center; use $\bu^0=0$ for the whole cube.
\State Set $L_n=\lceil n^{40}\rceil$, $\gamma_n=D_n/(B_n\sqrt{L_n})$, and $\bw=0$.
\For{$t=0,\ldots,L_n-1$}
    \State $\bw\gets\bw+\bu^t$.
    \State $\bg\gets\operatorname{arctanh}(\bu^t)-\bJ_{R_k,R_k}\bu^t-\widehat{\bh}^{(k)}(x_{S_k})+a_{k/n}\bu^t$.
    \State $\bu^{t+1}\gets\operatorname{Proj}_{\cK_k(x_{S_k})}(\bu^t-\gamma_n\bg)$.
    \State If a prescribed numerical cap or feasibility check fails, return a failure flag.
\EndFor
\State \Return $\widehat{\bm}^{(k)}(x_{S_k})=\bw/L_n$.
\end{algorithmic}
\end{breakablealgorithm}

\medskip
We now show that Algorithm~\ref{alg:tap-optimization} approximately optimizes $\Phi_k$ in $\cK_k$, as incorporated in the following lemma whose proof is postponed to Section~\ref{ss:deterministic-tap-optimization}.
\begin{lem}{\label{lem:optimization-accuracy}}
    Suppose that $\|\bJ\|_F\le n^2$ and $\Phi_k(\cdot;x_{S_k})$ is convex on $\cK_k(x_{S_k})$. Assume that Algorithm~\ref{alg:tap-optimization} returns a vector satisfying the prescribed precision and feasibility rules. Then, for all sufficiently large $n$,
    \begin{equation}\label{eq:sampling-optimization-accuracy}
        \Phi_k(\widehat{\bm}^{(k)}(x_{S_k});x_{S_k}) -\min_{\bu\in\cK_k(x_{S_k})}\Phi_k(\bu;x_{S_k})\le n^{-12}.
    \end{equation}
\end{lem}

Algorithm~\ref{alg:sequential-tap} implements the proposal oracle $\cA^{\pro}$, which samples from $Q$ and evaluates $Q$ at a supplied configuration. In \textsc{Sample} mode it draws the next sign; in \textsc{Evaluate} mode it uses the supplied sign and accumulates the same log-density. Each call starts its prefix computations afresh. In particular, the previous Metropolis state, previous proposal, and their checkpoint centers are not inputs. The initial stored vector, tuning record, order, and numerical rules remain fixed, but the phase orientation and checkpoint centers are recomputed from the configuration being generated or evaluated. 

\medskip
\begin{breakablealgorithm}
\caption{Sequential proposal generation or evaluation: \textsc{SequentialTAP}}
\label{alg:sequential-tap}
\begin{algorithmic}[1]
\Require $\bJ$, fixed order $I$, one tuning record with its stored $\bv$, and mode \textsc{Sample} or \textsc{Evaluate}.
\Require In \textsc{Evaluate} mode, a supplied configuration $s\in\{-1,1\}^n$.
\Ensure A configuration $s$ and its log-density $\ell_Q=\log Q(s)$.
\If{this record is assigned the uniform proposal}
    \State In \textsc{Sample} mode, draw $s$ uniformly; otherwise keep the supplied $s$.
    \State \Return $(s,-n\log2)$.
\EndIf
\State In \textsc{Sample} mode, draw $s_{I_1},\ldots,s_{I_{k_0}}$ independently and uniformly.
\State Set $\ell_Q\gets-k_0\log2$ and compute $\bv^+(s_{S_{k_0}})$ by \eqref{eq:sampling-phase-orientation}.
\State Form $R_{k_0}$ and $\widehat{\bh}^{(k_0)}(s_{S_{k_0}})$ from $s_{S_{k_0}}$; set $\mathsf{fallback}\gets\mathsf{false}$.
\State If orientation or field initialization fails, set $\mathsf{fallback}\gets\mathsf{true}$.
\For{$k=k_0,\ldots,n-1$}
    \If{$\mathsf{fallback}$}
        \State Set the entire vector $\widehat{\boldsymbol\ell}^{(k)}(s_{S_k})=0$.
    \Else
        \If{$k=j_\ell$ for some $0\le\ell<M$}
            \State Compute and store $\bv^{(k)}(s_{S_k})$ by \eqref{eq:sampling-side-amp}.
        \EndIf
        \State Form $\cK_k(s_{S_k})$ using \eqref{eq:sampling-checkpoint-domains}.
        \State $\widehat{\bm}^{(k)}(s_{S_k})\gets\textsc{TAPOptimize}(\bJ,s_{S_k},\widehat{\bh}^{(k)}(s_{S_k}),a_{k/n},\cK_k(s_{S_k}))$.
        \State Compute $\widehat{\boldsymbol\ell}^{(k)}(s_{S_k})$ by \eqref{eq:sampling-proposal-logits}.
        \State On any failure, set $\mathsf{fallback}\gets\mathsf{true}$ and the full logit vector to zero; skip unfinished computations.
    \EndIf
    \State Inspect $i=I_{k+1}$ and set $\ell=\widehat\ell^{(k)}_i(s_{S_k})$.
    \If{mode is \textsc{Sample}}
        \State Draw $s_i\in\{-1,1\}$ with $\bbP(s_i=+1)=(1+\tanh\ell)/2$.
    \EndIf
    \State $\ell_Q\gets\ell_Q+s_i\ell-\log(2\cosh\ell)$.
    \State $R_{k+1}\gets R_k\setminus\{i\}$.
    \If{not $\mathsf{fallback}$}
        \State $\widehat h_j^{(k+1)}(s_{S_{k+1}})\gets\widehat h_j^{(k)}(s_{S_k})+J_{ji}s_i$ for every $j\in R_{k+1}$.
        \State If this update fails, set $\mathsf{fallback}\gets\mathsf{true}$ for the next step.
    \EndIf
\EndFor
\State \Return $(s,\ell_Q)$.
\end{algorithmic}
\end{breakablealgorithm}
\medskip

\section{Analysis of the sequential proposal}\label{s:auxiliary-proofs}

We prove the domain and logit guarantees used in Section~\ref{ss:conditional-tap}. The local-convexity and state-evolution theorems are taken from the literature. The additional work is to transfer their conclusions along the revealing path and to obtain a quantitative overlap estimate when only $O(\log n)$ coordinates have been revealed.

\subsection{Scalar parameters and probabilistic inputs}\label{ss:prelim-lem}

We first record the scalar facts needed to choose the coefficients and checkpoints. For the Rademacher Gaussian channel, $F$ is increasing and strictly concave, $F(0)=0$, $F'(0)=1$, and
\[
    F'(u)=\bbE\operatorname{sech}^4(u+\sqrt u\,G)
       \le 1-F(u),\qquad
    1-F(u)<\frac1{1+u}\quad(u>0).
\]
These standard channel identities and the properties of the positive fixed point are given in \cite[Proposition~A.2 and its scalar-channel formulas]{celentano2023local}; the positive-side-channel version is also recorded in \cite[arXiv version, Proposition~3]{celentano2024sudakov}. The strict bound on $1-F(u)$ is the usual comparison with the linear estimator in the scalar Gaussian channel. In particular, $q_0>1-\lambda^{-2}$, and $q=F(rq+t)$ has a unique positive solution when $t>0$, with $r(1-q)<1$. Since $r_\alpha q_0+t_\alpha\ge\lambda^2q_0$, monotonicity of $F$ and uniqueness of the positive root give $q_\alpha\ge q_0$. Consequently,
\[
    q_\alpha\ge q_0>1-\lambda^{-2},\qquad
    0\le r_\alpha F'(r_\alpha q_\alpha+t_\alpha)
       \le a_\alpha\le a_0<1.
\]
Thus $[1-\lambda^{-2},1]$ is a valid bisection bracket. On this whole bracket, not only at the root,
\[
    \frac{\mathrm d}{\mathrm dq}\{q-F(r_\alpha q+t_\alpha)\}
       \ge 1-\lambda^2\{1-F(\lambda^2-1)\}>0.
\]
Indeed $r_\alpha q+t_\alpha\ge\lambda^2-1$, and the strict scalar bound above applies. This gives a positive derivative bound depending only on $\lambda$, which can itself be certified by quadrature. The implicit-function theorem applies uniformly on $\alpha\in[0,1]$. Differentiating the fixed-point equation gives
\[
    q_\alpha'
      =\frac{\lambda^2(1-q_\alpha)F'(r_\alpha q_\alpha+t_\alpha)}
             {1-r_\alpha F'(r_\alpha q_\alpha+t_\alpha)},\qquad
    a_\alpha'
      =-\frac{\lambda^2(1-q_\alpha)}
             {1-r_\alpha F'(r_\alpha q_\alpha+t_\alpha)}.
\]
The derivatives are bounded for fixed $\lambda>1$. In particular, with
$\bar q_\alpha=\alpha+(1-\alpha)q_\alpha$, there is $L_\lambda<\infty$ such that
\begin{equation}\label{eq:scalar-checkpoint-lipschitz}
    |a_\alpha-a_\beta|\le L_\lambda|\alpha-\beta|,\qquad
    0\le\bar q_\beta-\bar q_\alpha
       \le L_\lambda(\beta-\alpha),\qquad 0\le\alpha\le\beta\le1.
\end{equation}
Here $a_\alpha=\lambda^2(1-\bar q_\alpha)$ and
$\bar q_\alpha'=(1-q_\alpha)/(1-r_\alpha F'(r_\alpha q_\alpha+t_\alpha))>0$.
The same bounds hold uniformly when $\lambda$ ranges over a sufficiently small compact neighborhood of its fixed value in $(1,\infty)$.


For the probabilistic inputs, we use the following common normalization. Let $\theta$ be uniform on $\{-1,1\}^N$, let $\bW_N$ be GOE with off-diagonal variance $1/N$ and diagonal variance $2/N$, and let $\bg\sim\cN(0,\bI_N)$, independently. For $r,t\ge0$, set
\[
    \bA=\frac rN\theta\theta^{\top}+\sqrt r\,\bW_N,
    \qquad \bh=t\theta+\sqrt t\,\bg.
\]
The posterior distribution given $(\bA,\bh)$ is proportional to
$\exp\{\frac12\sigma^{\top}\bA\sigma+\bh^{\top}\sigma\}$ on $\{-1,1\}^N$.
The diagonal contribution is constant in $\sigma$, so retaining the GOE diagonal does not change the posterior or any of its overlaps. We write $\langle\cdot\rangle$ for posterior averaging and $\bbE$ for expectation under the joint planted law. Replicas $\sigma^1,\sigma^2,\ldots$ are independent posterior samples conditional on $(\bA,\bh)$, and
\[
    R_{a\theta}=\frac1N\sum_{i=1}^N\sigma_i^a\theta_i,
    \qquad R_{ab}=\frac1N\sum_{i=1}^N\sigma_i^a\sigma_i^b.
\]
Whenever $t>0$, or $t=0$ and $r>1$, $q$ denotes the unique positive solution of $q=F(rq+t)$. The matrix $\bA$ is generic in these inputs and will later be identified with a free block of $\bJ$.

We record the posterior identity, the two overlap-concentration inputs, and the finite-dimensional cavity expansion from \cite{alaoui2017finite} explicitly. The parameter comparisons needed to apply these results uniformly are given below; the interpolation and cavity theorems themselves are used as cited inputs.

\begin{lem}[Nishimori identity]\label{lem:nishimori-input}
Let $\cF$ be the sigma-field generated by the observations being retained. Conditional on $\cF$, let $\sigma^1,\ldots,\sigma^{a+1}$ be independent samples from $\mathcal L(\theta\mid\cF)$, independent also of $\theta$ conditionally on $\cF$. For every bounded function $f$ of $a+1$ configurations, which may also depend measurably on $\cF$,
\[
    \bbE\!\left[
       \langle f(\theta,\sigma^1,\ldots,\sigma^a)\rangle
          \,\middle|\,\cF\right]
       =\langle f(\sigma^{a+1},\sigma^1,\ldots,\sigma^a)\rangle.
\]
In particular, $R_{1\theta}$ and $R_{12}$ have the same distribution after averaging over the observations. The identity remains valid when the retained observations include revealed coordinates of $\theta$.
\end{lem}

This is the posterior-replica identity recorded in \cite[Section~3, ``The Nishimori property'']{alaoui2017finite}. In the Gaussian model above we take $\cF=\sigma(\bA,\bh)$. Along the revealing path we include the exposed index--value pairs in $\cF$ and use the posterior conditional on those pairs. Thus replacing the planted vector by a replica always refers to the posterior for the same retained observations.

\begin{lem}[Unpinned absolute-overlap input]\label{lem:unpinned-overlap-input}
In the Gaussian model above, let $t=0$ and $r>1$. For every $\varepsilon>0$ there are $C_{r,\varepsilon},c_{r,\varepsilon}>0$ such that
\[
    \bbE\left\langle
       \bone_{\{\left||R_{1\theta}|-q\right|>\varepsilon\}}
                    \right\rangle
       \le C_{r,\varepsilon}e^{-c_{r,\varepsilon}N}.
\]
The same estimate holds with $R_{1\theta}$ replaced by $R_{12}$. The constants can be chosen uniformly for $r$ in any fixed compact subinterval of $(1,\infty)$.
\end{lem}

For a fixed $r$, this is the last assertion of \cite[Proposition~13]{alaoui2017finite}, with their signal parameter equal to $r$ and their interpolation parameter equal to one. The replica formulation follows from Lemma~\ref{lem:nishimori-input}. The compact-parameter assertion follows from the comparison given after Lemma~\ref{lem:overlap-concentration}.

For the initial interval, define the unpinned TAP functional using the same interaction matrix as the sampler:
\begin{equation}{\label{eq:def-Psi-0}}
    \Psi_0(\bu)=\sum_{i=1}^n\phi(u_i) -\frac12\bu^{\top}\bJ\bu -\frac{\lambda^2n}{4} \left(1-\frac{\|\bu\|^2}{n}\right)^2, \qquad \bu\in(-1,1)^n.
\end{equation}

\begin{lem}[Unpinned TAP input]\label{lem:unpinned-tap-input}
Fix $\lambda>1$. There are an integer $L_{\rm init}\ge1$ and constants $\gamma_*,c_0,R_0>0$ such that, for any fixed $0<\gamma<\min\{\gamma_*,1\}$, the initial center $\bv$ constructed in Section~\ref{ss:conditional-tap} satisfies
\[
    \frac1n\left\|\bv\bv^{\top}
       -\bbE[\theta\theta^{\top}\mid\bJ]\right\|_F
       \xrightarrow{\bbP}0,\qquad
    \frac{\|\bv\|^2}{n}\xrightarrow{\bbP}q_0,\qquad
    \frac{|\langle\theta,\bv\rangle|}{n}\xrightarrow{\bbP}q_0.
\]
Moreover, with probability tending to one,
\[
    \nabla^2\Psi_0(\bu)\succeq c_0\bI
    \quad\text{for every }\bu\in(-1,1)^n\cap
       \bigl(B(\bv,R_0\sqrt n)\cup B(-\bv,R_0\sqrt n)\bigr).
\]
No planted sign is used to compute $\bv$.
\end{lem}

\begin{proof}
The local minimizer and its matrix-estimation and convexity guarantees are \cite[Theorem~2.1]{celentano2023local}; convergence from the spectral initialization followed by natural-gradient descent is \cite[Theorem~2.3]{celentano2023local}. To match conventions, their normalized observation is $\bJ/\lambda$ and $\Psi_0$ is $n$ times their TAP free energy. Their spectral initialization is exactly \eqref{eq:sampling-initial-amp}, including $\bu^{-1}=\lambda\bz^0$. Their geometric convergence estimate makes $n$ refinement steps sufficient. Shrinking the convexity radius absorbs both the refinement error and the clipping error $\sqrt n\,2^{-n}=o(1)$.

For the two scalar limits, put $\bB=\bbE[\theta\theta^{\top}\mid\bJ]$. Lemma~\ref{lem:unpinned-overlap-input}, with $r=\lambda^2$, and Lemma~\ref{lem:nishimori-input} give $\|\bB\|_F^2/n^2\to q_0^2$ and $\theta^{\top}\bB\theta/n^2\to q_0^2$ in probability, by conditional Jensen. The matrix-estimation guarantee then gives $\|\bv\|^2/n\to q_0$ and $\langle\theta,\bv\rangle^2/n^2\to q_0^2$. This also explains why the planted-overlap assertion is stronger than simply choosing a sign for an arbitrary matrix estimator.
\end{proof}

The later checkpoints use AMP in the same Gaussian model with $r,t>0$.

\begin{lem}[Positive-side-channel TAP input]\label{lem:side-amp-input}
Fix $r,t>0$ and consider the Gaussian model defined above, with $q=F(rq+t)$. Define
\[
    \Phi(\bu)=\sum_{i=1}^N\phi(u_i)
       -\frac12\bu^{\top}\bA\bu-\bh^{\top}\bu
       +\frac{r(1-q)}2\|\bu\|^2,
       \qquad \bu\in(-1,1)^N.
\]
Initialize $\bz^0=\bh$, $\bu^{-1}=0$, and use
\[
    \bu^\ell=\tanh\bz^\ell,\qquad
    \bz^{\ell+1}=\bA\bu^\ell+\bh
       -r\left(1-\frac{\|\bu^\ell\|^2}{N}\right)\bu^{\ell-1}.
\]
If $s_0=t$ and $s_{\ell+1}=rF(s_\ell)+t$, then, for every fixed $L\ge0$,
\[
    \frac{\|\bu^L\|^2}{N}\xrightarrow{\bbP}F(s_L),\qquad
    \frac{\langle\theta,\bu^L\rangle}{N}\xrightarrow{\bbP}F(s_L),
    \qquad F(s_L)\uparrow q\quad\text{as }L\to\infty.
\]
There are $c,R>0$, depending only on $r,t$, such that, for each fixed $0<\varepsilon<R$ and every sufficiently large fixed $L$,
\[
    \bbP\left(\nabla^2\Phi(\bu)\succeq c\bI_N
       \text{ for every }\bu\in(-1,1)^N\cap
                    B(\bu^L,\varepsilon\sqrt N)\right)\longrightarrow1.
\]
The order of quantifiers is that $\varepsilon$ and then $L$ are fixed before $N\to\infty$.
\end{lem}

This is a direct specialization of \cite[arXiv version, Proposition~3 and Theorem~2(i), AMS case]{celentano2024sudakov}. In that reference take the signal strength to be $\sqrt r$, the side-channel parameter to be $\gamma_0=t$, and the normalized observation to be $\bA/\sqrt r$. Its $q_\infty$ equals $q$, its iteration $m^{k-1}$ is our $\bu^L$ with $k=L+1$, and
\[
    \Phi(\bu)=N\mathcal F_{\mathrm{AMS}}(\bu)
                   +\frac{Nr(1-q)(1+q)}4.
\]
Thus the Hessian normalization is exactly the one stated here. The AMS case allows every $r>0$ and $t>0$; no assumption $r<1$ is needed. We use the theorem only at fixed positive side-channel strengths, not uniformly as $t\downarrow0$.

\begin{lem}[Positive-side-channel overlap input]\label{lem:overlap-concentration}
In the Gaussian model above, let $r\ge0$ and $t>0$. For every $\varepsilon>0$ there are $C_{r,t,\varepsilon},c_{r,t,\varepsilon}>0$ such that
\[
    \bbE\left\langle
       \bone_{\{|R_{1\theta}-q|>\varepsilon\}}\right\rangle
       \le C_{r,t,\varepsilon}e^{-c_{r,t,\varepsilon}N}.
\]
The same estimate holds with $R_{1\theta}$ replaced by $R_{12}$. For fixed $r_+<\infty$ and $0<t_-\le t_+<\infty$, the constants can be chosen uniformly over $0\le r\le r_+$ and $t_-\le t\le t_+$.
\end{lem}

The pointwise statement is \cite[Proposition~13]{alaoui2017finite}. To match its notation, set
\[
    \Lambda=r+\frac tq,\qquad \tau=\frac r\Lambda.
\]
Their signal parameter is $\Lambda$ and their interpolation parameter is $\tau$; our side-channel strength $t$ is not their interpolation parameter. Indeed,
\[
    q=F(\Lambda q),\qquad
    \tau\Lambda=r,\qquad (1-\tau)\Lambda q=t,
\]
so their Gaussian model in \cite[equations~(22)--(23)]{alaoui2017finite} has exactly the same posterior. Since $F(u)<u$ for $u>0$, we have $\Lambda>1$. The scalar facts above imply that the positive fixed point is the unique maximizer of the Rademacher replica-symmetric potential and varies smoothly for $\Lambda>1$. Consequently the twice-differentiability hypothesis in that reference is satisfied. The case $r=0$ corresponds to $\tau=0$ and is already included in Proposition~13. Thus no separate concentration theorem is needed at the product-measure endpoint. Lemma~\ref{lem:nishimori-input} gives the replica formulation. The signed exponential rate in Proposition~13 may tend to zero as $\tau\uparrow1$, which is why the uniform assertion here keeps $t$ bounded away from zero.

We justify the compact-parameter uniformity in both overlap lemmas without repeating their concentration proofs. We use the Gaussian norm estimates
\[
    \bbP(\|\bW_N\|_{\mathrm{op}}>10)
       +\bbP(\|\bg\|>2\sqrt N)\le Ce^{-cN},
\]
which follow from \cite[Corollary~5.35 and Proposition~5.34]{vershynin2012introduction}; for the matrix bound, write $\bW_N=(\bG+\bG^{\top})/\sqrt{2N}$ with independent standard Gaussian entries in $\bG$. Couple two parameter pairs $(r,t)$ and $(r',t')$ using the same $\theta,\bW_N,\bg$, and write $H_{r,t}(\sigma)=\frac12\sigma^{\top}\bA\sigma+\bh^{\top}\sigma$ for their Gibbs exponents. On the displayed norm event,
\[
\begin{split}
    \frac1N\sup_{\sigma\in\{-1,1\}^N}
        |H_{r,t}(\sigma)-H_{r',t'}(\sigma)|
    \le {}&\frac12|r-r'|+5|\sqrt r-\sqrt{r'}|\\
       &+|t-t'|+2|\sqrt t-\sqrt{t'}|.
\end{split}
\]
If the right-hand side is at most $\eta$, the ratio of the corresponding normalized Gibbs densities is at most $e^{2\eta N}$. The overlap with the common planted vector is the same function of $\sigma$ in the two models. Continuity of the positive fixed point therefore lets us transfer the exponential estimate at a fixed parameter pair, with tolerance $\varepsilon/2$, to all sufficiently close parameter pairs with tolerance $\varepsilon$. Choose the neighborhood so that the change in $q$ is less than $\varepsilon/2$ and $2\eta$ is less than half the exponential rate at its center. The complement of the norm event contributes at most $Ce^{-cN}$. A finite cover gives the uniformity in Lemma~\ref{lem:overlap-concentration}, including $r=0$. The same argument with $t=t'=0$ gives the uniformity in Lemma~\ref{lem:unpinned-overlap-input}. In each comparison the expectations are taken through the common planted/noise coupling, so both coupled models have their correct planted laws.

The last input is the cavity expansion before overlap concentration is used to control its remainder. This form is useful because its constants remain bounded when the side channel vanishes.

\begin{lem}[Finite-dimensional cavity moment input]\label{lem:cavity-moment-input}
Consider the Gaussian model above with $q>0$, and suppose that $\Lambda=r+t/q$ belongs to a fixed compact subinterval of $(1,\infty)$. Set
\[
    p=\bbE\tanh^4(rq+t+\sqrt{rq+t}\,G),
    \qquad G\sim\cN(0,1),
\]
and define
\[
    \bz=
    \begin{pmatrix}
       \bbE\langle(R_{1\theta}-q)^2\rangle\\
       \bbE\langle(R_{1\theta}-q)(R_{2\theta}-q)\rangle\\
       \bbE\langle(R_{1\theta}-q)(R_{23}-q)\rangle
    \end{pmatrix},
\]
\[
    \bC_{\rm cav}=r
    \begin{pmatrix}
      1-q^2&-2(q-q^2)&p-q^2\\
      q-q^2&1-q-2p+2q^2&-2q+3p-q^2\\
      p-q^2&4q-6p+2q^2&1-6q+6p-q^2
    \end{pmatrix}.
\]
There is a remainder vector $\boldsymbol\delta\in\bbR^3$ such that
\begin{equation}\label{eq:eakj-cavity-system}
    (\bI-\bC_{\rm cav})\bz
       =\frac1N
           \begin{pmatrix}1-q^2\\q-q^2\\p-q^2\end{pmatrix}
          +\boldsymbol\delta,
\end{equation}
and
\begin{equation}\label{eq:eakj-cavity-remainder}
    \|\boldsymbol\delta\|_{\infty}
       \le C\Big\{ \bbE\langle|R_{12}-q|^3\rangle
           +\frac1N\bbE\langle(R_{12}-q)^2\rangle
           +\frac1N\bbE\langle|R_{12}-q|\rangle
           +\frac1{N^2}\Big\}.
\end{equation}
Here $C$ depends only on the compact interval for $\Lambda$, uniformly in $N$ and in $\tau=r/\Lambda\in[0,1]$. The eigenvalues of $\bC_{\rm cav}$ are $r(1-2q+p)$, with algebraic multiplicity two, and $r(1-3q+2p)$.
\end{lem}

This is the Rademacher specialization of \cite[Sections~5.2--5.3, equations~(40)--(41), (46), and~(48)]{alaoui2017finite}. In their notation the three scalar coefficients are $a(0)=1-q^2$, $a(1)=q-q^2$, and $a(2)=p-q^2$; their matrix multiplied by $\tau$ is our $\bC_{\rm cav}$. The parameter identification is the same as above. Lemma~\ref{lem:nishimori-input} permits the use of $R_{12}$ in the remainder and identifies the first component of $\bz$ with $\bbE\langle(R_{12}-q)^2\rangle$. The interpolation comparisons and Taylor bounds in \cite[Lemmas~15--17]{alaoui2017finite} have constants depending only on an upper bound for $\Lambda$, the bounded prior support, and the fixed number of replicas. This gives the uniformity asserted here, including $\tau=1$. We use the cited finite-dimensional identities directly; no signed-overlap concentration is assumed in this input.

\subsection{Strong convexity along the revealing path}\label{ss:checkpoint-construction}

In this subsection, we prove Lemma~\ref{lem:convexity-barycenter} by choosing a valid tuning record for the constraint sets $\cK_k$ specified in Section~\ref{ss:conditional-tap} and establishing their simultaneous guarantees. All probabilities in this proof are under the joint planted law, including the independent revealing order. 

We begin by identifying the conditional model at a deterministic checkpoint and recording its posterior-mean limit. By Lemma~\ref{lem:overlap-concentration}, the overlap of two independent posterior replicas concentrates at $q$ in the fixed positive-side-channel model. Conditional Jensen's inequality and boundedness then imply
\begin{equation}\label{eq:fixed-side-posterior-norm}
    \frac{\|\bm\|^2}{N} \overset{L^2}{\longrightarrow} q.
\end{equation}
Indeed, the left-hand side is the conditional expectation of the overlap of two independent posterior replicas. Combining \eqref{eq:fixed-side-posterior-norm} with Lemma~\ref{lem:side-amp-input}, for each fixed AMP count $L$ we also obtain (for the cross term below, use $\langle\bu^L,\bm\rangle/N=\bbE[\langle\bu^L,\theta\rangle/N\mid\bA,\bh]$)
\begin{equation}\label{eq:fixed-side-amp-mean-distance}
    \frac{\|\bu^L-\bm\|^2}{N} \overset{\bbP}{\longrightarrow} q-F(s_L).
\end{equation}
The state-evolution limit of the variable inside this conditional expectation also holds in $L^2$, because it is bounded. At a deterministic checkpoint $j$, conditional on $I$ and $\theta_{S_j}$, the pair $(\bJ_{R_j,R_j},\bh^{(j)}(\theta_{S_j}))$ has the law of $(\bA,\bh)$ with $N=n-j$, $r=r_{j/n}$, and $t=t_{j/n}=\lambda^2j/n$. In fact,
\[
    h_i^{(j)}(\theta_{S_j})=\lambda^2\frac jn\theta_i
       +\lambda\sum_{a=1}^jW_{i,I_a}\theta_{I_a},\qquad i\in R_j.
\]
The Gaussian sums have variance $\lambda^2j/n$, are independent across $i$, and are independent of the GOE block. Moreover, the conditional Gibbs density depends on the observed data only through this block and this field. Thus the conditional mean appearing in \eqref{eq:fixed-side-posterior-norm} is precisely $\bm^{(j)}(\theta_{S_j})$. 
We also justify the rounding $j=\lfloor\alpha n\rfloor$ for a fixed $\alpha\in(0,1)$. Couple the parameters $(r_{j/n},t_{j/n})$ and $(r_\alpha,t_\alpha)$ using the same $\theta,\bW_N,\bg$. On $\|\bW_N\|_{\mathrm{op}}\le10$ and $\|\bg\|\le2\sqrt N$, the matrix difference has operator norm $O(n^{-1})$ and the field difference has norm $O(\sqrt N/n)$. Consequently the two Hamiltonians differ by $O(N/n)=O(1)$ uniformly over the spin configurations, and the ratio of their normalized Gibbs densities is bounded by a constant. The complement of this norm event has probability tending to zero. The fixed-parameter overlap concentration therefore transfers to the rounded model, and $q_{j/n}-q_\alpha\to0$. For every fixed number of AMP updates, the coupled iterates differ by $O(\sqrt N/n)$ on the same norm event, by induction in the recursion using the Lipschitz property of $\tanh$. Hence both \eqref{eq:fixed-side-posterior-norm} and \eqref{eq:fixed-side-amp-mean-distance} apply at rounded checkpoints. For the rest of the proof, extend the conditional mean to all coordinates by setting, for any prefix,
\[
    \overline m_i^{(k)}(x_{S_k})=
    \begin{cases} 
        x_i,&i\in S_k,\\ 
        m_i^{(k)}(x_{S_k}),&i\in R_k.
    \end{cases}
\]
Along the target path, $\overline{\bm}^{(k)}(\theta_{S_k})=\bbE[\theta\mid\cH_k]$. Conditioning additionally on the unused part of the independent permutation does not change this posterior mean. Thus $(\overline{\bm}^{(k)}(\theta_{S_k}))_{k=0}^n$ is a vector-valued martingale, and for every fixed $\alpha\in(0,1)$,
\begin{equation}\label{eq:fixed-prefix-mean-norm}
    \frac{\|\overline{\bm}^{(\lfloor\alpha n\rfloor)}(\theta_{S_{\lfloor\alpha n\rfloor}})\|^2}{n} \overset{L^2}{\longrightarrow} \bar q_\alpha.
\end{equation}

To pass from deterministic checkpoints to the whole revealing path, we use the following maximal estimate, an elementary consequence of the posterior-mean martingale. Let $p=p_n\le h=h_n$ be deterministic indices and let $\bw\in[-1,1]^n$ be $\cH_p$-measurable. Suppose, for deterministic constants $c,d$, that
\[
    \frac{\|\bw\|^2}{n} \overset{L^2}{\longrightarrow} c,\qquad \frac{\langle\theta,\bw\rangle}{n} \overset{L^2}{\longrightarrow} c, \qquad \frac{\|\overline{\bm}^{(h)}(\theta_{S_h})\|^2}{n} \overset{L^2}{\longrightarrow} d.
\]
Then
\begin{equation}\label{eq:checkpoint-maximal-distance}
    \max_{p\le k\le h} \frac{\|\overline{\bm}^{(k)}(\theta_{S_k})-\bw\|^2}{n} \le d-c+o_{\bbP}(1).
\end{equation}
To prove this, conditional Jensen gives $\|\overline{\bm}^{(k)}(\theta_{S_k})\|^2/n \le\bbE[\|\overline{\bm}^{(h)}(\theta_{S_h})\|^2/n\mid\cH_k]$ for $k\le h$. Since $\bw$ is $\cH_p$-measurable, $\langle\overline{\bm}^{(k)}(\theta_{S_k}),\bw\rangle/n =\bbE[\langle\theta,\bw\rangle/n\mid\cH_k]$ for $k\ge p$. Doob's $L^2$ maximal inequality yields
\[
\begin{split}
    \bbE\max_{0\le k\le h} \left|\bbE\left[\frac{\|\overline{\bm}^{(h)}(\theta_{S_h})\|^2}{n}-d \mid \cH_k\right]\right|^2 &\le4\bbE\left|\frac{\|\overline{\bm}^{(h)}(\theta_{S_h})\|^2}{n}-d \right|^2=o(1),\\
    \bbE\max_{p\le k\le n} \left|\bbE\left[\frac{\langle\theta,\bw\rangle}{n}-c \mid \cH_k\right]\right|^2 &\le4\bbE\left| \frac{\langle\theta,\bw\rangle}{n}-c \right|^2=o(1).
\end{split}
\]
Expanding $\|\overline{\bm}^{(k)}(\theta_{S_k})-\bw\|^2$ proves
\eqref{eq:checkpoint-maximal-distance}. This estimate is uniform over every time within the interval, without a union bound over those times.


\begin{proof}[Proof of Lemma~\ref{lem:convexity-barycenter}]
We choose the tuning record in the course of the proof; every parameter and every AMP iteration count is then held fixed as $n\to\infty$. We divide the proof into four steps.

\textbf{Step 1: Choose $L_{\rm init},\rho_0,\alpha_0$ and deal with $k_0 \le k \le j_0$.} We apply the estimate \eqref{eq:checkpoint-maximal-distance} to the initial center and the interval before the first checkpoint. Choose $L_{\rm init}$ and $\gamma$ according to Lemma~\ref{lem:unpinned-tap-input}, and let $\bv$ be its stored vector. By that lemma, $\|\bv\|^2/n\to q_0$ and $|\langle\theta,\bv\rangle|/n\to q_0$ in probability. Conditional on $(\bJ,\theta,\bv)$, the first $k_0$ indices form a uniform sample without replacement. On $|\langle\theta, \bv\rangle|/n \ge q_0/2$, the sample average of $\operatorname{sign}(\langle\theta,\bv\rangle)\theta_i v_i$ has mean at least $q_0/2$ and variance at most $1/k_0$. Therefore the probability of a wrong or zero phase orientation is at most $4/(k_0q_0^2)$. Since $k_0\to\infty$, the prefix-measurable orientation in \eqref{eq:sampling-phase-orientation} satisfies
\begin{equation}\label{eq:oriented-initial-center-limits}
    \frac{\|\bv^+(\theta_{S_{k_0}})\|^2}{n} \overset{L^2}{\longrightarrow} q_0, \qquad \frac{\langle\theta,\bv^+(\theta_{S_{k_0}})\rangle}{n} \overset{L^2}{\longrightarrow} q_0.
\end{equation}
Here boundedness upgrades the convergence in probability to $L^2$. For this comparison, define the unpinned objective
\[
    \Phi_0(\bu)=\sum_{i=1}^n\phi(u_i)
       -\frac12\bu^\top\bJ\bu+\frac{a_0}{2}\|\bu\|^2,
       \qquad \bu\in(-1,1)^n.
\]
It is the extension of \eqref{eq:def-Phi-k} to no revealed coordinates and has no prefix argument.
Recall \eqref{eq:def-Psi-0}. Direct differentiation gives
\[
    \nabla^2\Phi_0(\bu)-\nabla^2\Psi_0(\bu) =\lambda^2\left(\frac{\|\bu\|^2}{n}-q_0\right)\bI +\frac{2\lambda^2}{n}\bu\bu^{\top}.
\]
The rank-one term is positive semidefinite. On
$B(\bv^+(\theta_{S_{k_0}}),2\rho_0\sqrt n)\cap(-1,1)^n$,
\[
    \left|\frac{\|\bu\|^2}{n}-q_0\right| \le4\rho_0+ \left|\frac{\|\bv^+(\theta_{S_{k_0}})\|^2}{n}-q_0\right|.
\]
Choose a fixed $\rho_0>0$ so that $2\rho_0<R_0$ and
$4\lambda^2\rho_0<c_0/4$, where $c_0,R_0$ are from
Lemma~\ref{lem:unpinned-tap-input}. That lemma and
\eqref{eq:oriented-initial-center-limits} imply, with probability tending to one,
\[
    \nabla^2\Phi_0(\bu)\succeq\frac{c_0}{2}\bI \quad\text{for every }\bu\in B(\bv^+(\theta_{S_{k_0}}),2\rho_0\sqrt n)\cap(-1,1)^n.
\]
Choose a fixed $\alpha_0>0$ small enough that (recall the definition of $L_\lambda$ in \eqref{eq:scalar-checkpoint-lipschitz})
\[
    L_\lambda\alpha_0<\rho_0^2/16, \qquad L_\lambda\alpha_0<c_0/4,
\]
and put $j_0=\lfloor\alpha_0n\rfloor$. Apply \eqref{eq:checkpoint-maximal-distance} with $p=k_0$, $h=j_0$, $\bw=\bv^+(\theta_{S_{k_0}})$, $c=q_0$, and $d=\bar q_{\alpha_0}$. Its hypotheses follow from \eqref{eq:fixed-prefix-mean-norm} and \eqref{eq:oriented-initial-center-limits}. Restricting the resulting distances to $R_k$ gives, with probability tending to one,
\[
    \max_{k_0\le k\le j_0} \frac{\|\bm^{(k)}(\theta_{S_k})-\bv^+(\theta_{S_{k_0}})|_{R_k}\|^2}{n} \le L_\lambda\alpha_0+o_{\bbP}(1)<\rho_0^2/4.
\]
Thus the first ball contains the conditional mean simultaneously throughout the initial interval. 

\textbf{Step 2: Choose $M,\rho,\kappa_\lambda,\alpha_{\ell+1}$ for $0 \le \ell \le M-1$.} We first fix a terminal fraction $\alpha_M\in(\alpha_0,1)$ sufficiently close to one that
\[
    \lambda^2(1-\alpha_M)+10\lambda\sqrt{1-\alpha_M}<1/4.
\]
At this point $\alpha_M$ denotes the terminal fraction; its index $M$, the number of checkpoint intervals, will be chosen at the end of this step. This choice will give whole-cube convexity on the final suffix in Step~4. On the intermediate interval $[\alpha_0,\alpha_M]$, both $r_\alpha=\lambda^2(1-\alpha)$ and $t_\alpha=\lambda^2\alpha$ are bounded away from zero. We will first obtain a common convexity radius around the exact posterior means on this interval, and only then choose the mesh and the final AMP iteration counts.

For the parameter comparison, couple the positive-side-channel models in a common dimension $N$, using the same independent $\theta,\bW_N,\bg$ as in Lemma~\ref{lem:side-amp-input}. Displaying the fraction parameter explicitly, set
\[
    \bA_\beta=\frac{r_\beta}{N}\theta\theta^{\top}
                    +\sqrt{r_\beta}\,\bW_N,
    \qquad
    \bh_\beta=t_\beta\theta+\sqrt{t_\beta}\,\bg,
    \qquad
    \bm_\beta=\bbE[\theta\mid\bA_\beta,\bh_\beta].
\]
Thus $\bm_\beta$ is the conditional expectation with respect to $\sigma(\bA_\beta,\bh_\beta)$, not with respect to the observations from all the coupled models. Write
\[
    \Phi_\beta(\bu)=\sum_{i=1}^N\phi(u_i)
        -\frac12\bu^{\top}\bA_\beta\bu-\bh_\beta^{\top}\bu
        +\frac{a_\beta}{2}\|\bu\|^2,
        \qquad \bu\in(-1,1)^N,
\]
and let $\bu_\beta^L=\tanh\bz_\beta^L$ denote its AMP iterate, initialized with $\bz_\beta^0=\bh_\beta$ and $\bu_\beta^{-1}=0$. Its scalar recursion is
\[
    s_0(\beta)=t_\beta,\qquad s_{L+1}(\beta)=r_\beta F(s_L(\beta))+t_\beta.
\]
This is an auxiliary coupling in a fixed dimension, not a comparison of two actual free-coordinate sets of different sizes. All parameter comparisons below take place in a fixed compact subinterval of $(0,1)$ whose interior contains $[\alpha_0,\alpha_M]$. We claim that, for every $\alpha\in[\alpha_0,\alpha_M]$, there are constants $c_\alpha,\varepsilon_\alpha,\delta_\alpha>0$ and a fixed integer $L_\alpha\ge1$ such that, for each fixed $\beta\in(\alpha-\delta_\alpha,\alpha+\delta_\alpha)$, the following properties hold jointly with probability tending to one as $N\to\infty$:
\begin{enumerate}
    \item[(I)] The Hessian satisfies $\nabla^2\Phi_\beta(\bu)\succeq\frac{c_\alpha}{2}\bI_N$ for every $\bu\in(-1,1)^N \cap B(\bu_\beta^{L_\alpha},6\varepsilon_\alpha\sqrt N)$.
    \item[(II)] The AMP iterate satisfies $\|\bu_\beta^{L_\alpha}-\bm_\beta\| <\frac{\varepsilon_\alpha}{4}\sqrt N$.
\end{enumerate}
Here the parameter neighborhood is chosen inside $(0,1)$. The constants and the iteration count are common to the neighborhood, but the probability statement is for each fixed $\beta$; we do not assert a simultaneous state-evolution estimate over a continuum of fractions. To prove the claim, fix a fraction $\alpha$. Lemma~\ref{lem:side-amp-input} gives constants $c_\alpha,R_\alpha>0$ such that, for each fixed radius smaller than $R_\alpha$, the local-convexity conclusion holds for every sufficiently large fixed AMP count. Choose $\varepsilon_\alpha>0$ with $8\varepsilon_\alpha<R_\alpha$. By \eqref{eq:fixed-side-amp-mean-distance},
\[
    \frac{\|\bu_\alpha^L-\bm_\alpha\|^2}{N} \overset{\bbP}{\longrightarrow} q_\alpha-F(s_L(\alpha)), \qquad F(s_L(\alpha))\uparrow q_\alpha\quad(L\to\infty).
\]
We may therefore fix $L_\alpha$ sufficiently large that the local-convexity conclusion applies at radius $8\varepsilon_\alpha\sqrt N$ and
\[
    q_\alpha-F(s_{L_\alpha}(\alpha)) <\varepsilon_\alpha^2/256.
\]
With probability tending to one, we then have both
\[
    \nabla^2\Phi_\alpha(\bu)\succeq c_\alpha\bI_N \text{ for every }\bu \in(-1,1)^N \cap B(\bu_\alpha^{L_\alpha},8\varepsilon_\alpha\sqrt N), \qquad \|\bu_\alpha^{L_\alpha}-\bm_\alpha\| <\varepsilon_\alpha\sqrt N/8.
\]
In this choice, the radius is fixed before the iteration count. We next transfer the anchor's convexity bound to nearby parameters. On the common norm event $\|\bW_N\|_{\mathrm{op}}\le10$ and $\|\bg\|\le2\sqrt N$, whose probability tends to one by the Gaussian bounds recorded above,
\[
\begin{aligned}
    \|\bA_\alpha-\bA_\beta\|_{\mathrm{op}}
       &\le |r_\alpha-r_\beta|
                  +10|\sqrt{r_\alpha}-\sqrt{r_\beta}|
        \le C|\alpha-\beta|,\\
    \|\bh_\alpha-\bh_\beta\|
       &\le\bigl(|t_\alpha-t_\beta|
                    +2|\sqrt{t_\alpha}-\sqrt{t_\beta}|\bigr)\sqrt N
        \le C|\alpha-\beta|\sqrt N.
\end{aligned}
\]
The constants here and in the following fixed-iteration comparison are independent of $N$. The square-root estimates use that $r_\beta$ and $t_\beta$ stay bounded away from zero on the chosen compact parameter interval. Since $\tanh$ is $1$-Lipschitz and $\|\bu_\beta^t\|\le\sqrt N$, subtracting the two AMP recursions and applying induction gives
\[
    \|\bu_\alpha^{L_\alpha}-\bu_\beta^{L_\alpha}\| \le C_{L_\alpha}|\alpha-\beta|\sqrt N.
\]
The constant may depend on $L_\alpha$, which is harmless because $L_\alpha$ is already fixed and the parameter neighborhood is chosen afterward. We now compare the Hessians without estimating the variation of the entropy Hessian near $\pm1$. Indeed, at the same vector $\bu\in(-1,1)^N$,
\[
\begin{aligned}
    \nabla^2\Phi_\beta(\bu) &=\operatorname{diag}\left(\frac1{1-u_i^2}\right)_{i=1}^N -\bA_\beta+a_\beta\bI_N, \\
    \nabla^2\Phi_\beta(\bu)-\nabla^2\Phi_\alpha(\bu) &=-(\bA_\beta-\bA_\alpha)+(a_\beta-a_\alpha)\bI_N.
\end{aligned}
\]
The entropy terms cancel exactly, and the fields contribute only linear terms. By \eqref{eq:scalar-checkpoint-lipschitz} and the matrix comparison above,
\[
    \sup_{\bu\in(-1,1)^N} \|\nabla^2\Phi_\beta(\bu)-\nabla^2\Phi_\alpha(\bu)\|_{\mathrm{op}} \le C|\alpha-\beta|.
\]
Shrink $\delta_\alpha$ so that, whenever $|\alpha-\beta|<\delta_\alpha$, the AMP centers differ by at most $\varepsilon_\alpha\sqrt N$ and this Hessian difference has operator norm at most $c_\alpha/2$ on the norm event. If $\|\bu-\bu_\beta^{L_\alpha}\|\le6\varepsilon_\alpha\sqrt N$, then
\[
    \|\bu-\bu_\alpha^{L_\alpha}\|
      \le6\varepsilon_\alpha\sqrt N
          +\|\bu_\beta^{L_\alpha}-\bu_\alpha^{L_\alpha}\|
      \le7\varepsilon_\alpha\sqrt N<8\varepsilon_\alpha\sqrt N.
\]
Thus the ball around the new center is contained in the anchor's convexity ball, and the Hessian perturbation costs at most $c_\alpha/2$. This proves (I). To prove (II), note that $\beta\mapsto q_\beta-F(s_{L_\alpha}(\beta))$ is continuous. Shrink $\delta_\alpha$ again so that
\[
    q_\beta-F(s_{L_\alpha}(\beta))<\varepsilon_\alpha^2/64 \quad\text{whenever }|\alpha-\beta|<\delta_\alpha.
\]
For each fixed $\beta$ in this neighborhood, \eqref{eq:fixed-side-amp-mean-distance} gives
\[
    \frac{\|\bu_\beta^{L_\alpha}-\bm_\beta\|^2}{N}
        \overset{\bbP}{\longrightarrow}
            q_\beta-F(s_{L_\alpha}(\beta)),
\]
which implies (II), with the stated strict margin. This argument does not require continuity of the random posterior mean as the model parameter varies. We use continuity of the fixed AMP iterates to move the convexity ball, and continuity of the scalar limiting error to locate the posterior mean inside that ball. On the intersection of (I) and (II), every $\bu$ with $\|\bu-\bm_\beta\|\le4\varepsilon_\alpha\sqrt N$ satisfies
\[
    \|\bu-\bu_\beta^{L_\alpha}\|
       \le\|\bu-\bm_\beta\|
                    +\|\bm_\beta-\bu_\beta^{L_\alpha}\|
       <\left(4+\frac14\right)\varepsilon_\alpha\sqrt N
       <6\varepsilon_\alpha\sqrt N.
\]
Consequently, for each fixed $\beta$ in this neighborhood, with probability tending to one,
\[
    \nabla^2\Phi_\beta(\bu)\succeq\frac{c_\alpha}{2}\bI_N
       \quad\text{for every }\bu\in(-1,1)^N
                 \cap B(\bm_\beta,4\varepsilon_\alpha\sqrt N).
\]
The conclusion is now centered at the exact posterior mean and no longer involves the auxiliary iteration count $L_\alpha$. This is what allows the algorithm's final AMP counts to be chosen later.

The open intervals $(\alpha-\delta_\alpha,\alpha+\delta_\alpha)$ cover $[\alpha_0,\alpha_M]$. Choose a finite subcover, and let $\varepsilon_*>0$ and $c_*>0$ be the minima of the corresponding finitely many radii $\varepsilon_\alpha$ and constants $c_\alpha$. For every fixed $\beta\in[\alpha_0,\alpha_M]$, select one of these neighborhoods containing $\beta$. The preceding conclusion gives, with probability tending to one,
\[
    \nabla^2\Phi_\beta(\bu)\succeq\frac{c_*}{2}\bI_N \quad\text{for every }\bu\in(-1,1)^N \cap B(\bm_\beta,4\varepsilon_*\sqrt N).
\]
We now apply this conclusion at an actual checkpoint $j=\lfloor\beta n\rfloor$, with $N=n-j$. Fix a selected neighborhood containing $\beta$; it also contains $j/n$ for all sufficiently large $n$. The coupled-iterate and same-argument Hessian comparisons above are deterministic on the common norm event, so they apply to $j/n$ as well. More explicitly, the extra perturbations from $\beta$ to $j/n$ are $O(\sqrt N/n)$ for a fixed AMP iterate and $O(n^{-1})$ for the Hessian. The rounding argument preceding the proof supplies the needed AMP-to-mean estimate in this $n$-dependent model: for the neighborhood's fixed count $L_\alpha$,
\[
    \frac{\|\bu_{j/n}^{L_\alpha}-\bm_{j/n}\|^2}{N}
       \overset{\bbP}{\longrightarrow}
            q_\beta-F(s_{L_\alpha}(\beta))
       <\varepsilon_\alpha^2/64.
\]
Under the conditional-model identification above, $\bm_{j/n}$ corresponds to $\bm^{(j)}(\theta_{S_j})$ and $\Phi_{j/n}$ to $\Phi_j(\cdot;\theta_{S_j})$. Since $N/n\ge1-\alpha_M$, choose constants
\[
    0<\rho<\frac{\varepsilon_*}{2}\sqrt{1-\alpha_M},
    \qquad 0<\kappa_1<c_*/16.
\]
Returning to the original coordinates, we have shown that, for every fixed $\beta\in[\alpha_0,\alpha_M]$ and $j=\lfloor\beta n\rfloor$, with probability tending to one,
\begin{equation}\label{eq:checkpoint-mean-convexity}
\begin{gathered}
    \nabla_{\bu}^2\Phi_j(\bu;\theta_{S_j})\succeq4\kappa_1\bI
       \quad\text{for every }\bu\in(-1,1)^{R_j}\cap
             B(\bm^{(j)}(\theta_{S_j}),4\rho\sqrt n).
\end{gathered}
\end{equation}
The constants are common to all fixed fractions in the interval; the high-probability assertion remains pointwise in the fraction. Once a finite mesh is chosen, its checkpoint events can therefore be intersected with probability tending to one by a finite union bound, regardless of their dependence. With $\rho$ and $\kappa_1$ fixed, choose $\Delta>0$ so that
\[
    L_\lambda\Delta<\rho^2/32,\qquad
    L_\lambda\Delta<\kappa_1.
\]
Choose an integer $M\ge1$ large enough that $(\alpha_M-\alpha_0)/M\le\Delta$, and set $\alpha_\ell=\alpha_0+\ell(\alpha_M-\alpha_0)/M$ for $1\le\ell<M$. This fixes a finite deterministic mesh independently of $n$. The first bound on $\Delta$ will control the motion of the posterior mean between checkpoints, and the second will control the loss of curvature from the change in $a_{k/n}$. Finally, set $\kappa_\lambda=\min\{c_0/4,\kappa_1,1/2\}$. The auxiliary counts $L_\alpha$ used in the finite-cover argument need not be the algorithm's checkpoint counts $L_\ell$; those will now be chosen using the already fixed radius $\rho$.

\textbf{Step 3: Choose $L_\ell$ and deal with $j_{\ell} \le k \le j_{\ell+1}$ for $0 \le \ell \le M-1$.} Use the mesh fixed in Step~2. At each checkpoint $j_\ell=\lfloor\alpha_\ell n\rfloor$, $0\le\ell<M$, choose a sufficiently large fixed AMP count $L_\ell$ so that
\[
    (1-\alpha_\ell)\{q_{\alpha_\ell}-F(s_{L_\ell}(\alpha_\ell))\}<\rho^2/64,
\]
where $s_L(\alpha_\ell)$ is the scalar recursion defined in Step~2. By \eqref{eq:fixed-side-amp-mean-distance}, its rounded-checkpoint version, and the clipping bound below, the stored center is within $\rho\sqrt n/4$ of the posterior mean with probability tending to one. On this event,
\[
    B(\bv^{(j_\ell)}(\theta_{S_{j_\ell}}),2\rho\sqrt n)
       \subset B(\bm^{(j_\ell)}(\theta_{S_{j_\ell}}),4\rho\sqrt n).
\]
Combining this inclusion with \eqref{eq:checkpoint-mean-convexity}, the clipped center $\bv^{(j_\ell)}(\theta_{S_{j_\ell}})$ then satisfies, with probability tending to one,
\begin{equation}\label{eq:checkpoint-center-margins}
\begin{split}
    \|\bv^{(j_\ell)}(\theta_{S_{j_\ell}})-\bm^{(j_\ell)}(\theta_{S_{j_\ell}})\|&<\rho\sqrt n/4,\\
    \nabla_{\bu}^2\Phi_{j_\ell}(\bu;\theta_{S_{j_\ell}})& \succeq4\kappa_1\bI \quad\text{for all }\bu\in(-1,1)^{R_{j_\ell}} \cap B(\bv^{(j_\ell)}(\theta_{S_{j_\ell}}),2\rho\sqrt n).
\end{split}
\end{equation}
Clipping changes the center by at most $\sqrt n\,2^{-n}=o(1)$, which is absorbed by the strict margins. Increasing the final AMP count does not require a new convexity argument: \eqref{eq:checkpoint-mean-convexity} is centered at the posterior mean and is independent of that count.

For the mean-inclusion assertion between these checkpoints, fix $\ell$ and extend $\bv^{(j_\ell)}(\theta_{S_{j_\ell}})$ to a vector $\bw\in[-1,1]^n$ by putting $w_i=\theta_i$ on $S_{j_\ell}$. It is $\cH_{j_\ell}$-measurable. State evolution, including the negligible clipping error, gives
\[
    \frac{\|\bw\|^2}{n} \overset{L^2}{\longrightarrow} c,\qquad \frac{\langle\theta,\bw\rangle}{n} \overset{L^2}{\longrightarrow} c, \qquad c=\alpha_\ell+(1-\alpha_\ell)F(s_{L_\ell}(\alpha_\ell)).
\]
Apply \eqref{eq:checkpoint-maximal-distance} with $p=j_\ell$, $h=j_{\ell+1}$, and $d=\bar q_{\alpha_{\ell+1}}$, using \eqref{eq:fixed-prefix-mean-norm} at the right endpoint. It follows that
\[
\begin{split}
    \max_{j_\ell\le k\le j_{\ell+1}} \frac{\|\bm^{(k)}(\theta_{S_k})-\bv^{(j_\ell)}(\theta_{S_{j_\ell}})|_{R_k}\|^2}{n}
    &\le \bar q_{\alpha_{\ell+1}}-\bar q_{\alpha_\ell} +(1-\alpha_\ell)\{q_{\alpha_\ell}-F(s_{L_\ell}(\alpha_\ell))\} +o_{\bbP}(1)\\
    &<\rho^2/4
\end{split}
\]
with probability tending to one. There are only finitely many intervals, so all these mean-inclusion events and \eqref{eq:checkpoint-center-margins} may be intersected. The mean-inclusion bounds are now in place. We next verify that the Hessian remains positive on the entire constraint set between successive checkpoints.
For any input prefix $x_{S_k}$, the modified Hessian has the form
\[
    \nabla_{\bu}^2\Phi_k(\bu;x_{S_k})
       =\operatorname{diag}\left(\frac1{1-u_i^2}\right)_{i\in R_k}
          -\bJ_{R_k,R_k}+a_{k/n}\bI.
\]
This Hessian does not depend explicitly on the prefix, but the region $\cK_k(x_{S_k})$ on which positivity is needed does, through its stored center. We now return to the target inputs. Fix a checkpoint $j$ and let $\bv^{(j)}(\theta_{S_j})$ be its stored center; for the initial interval use $j=0$ and the oriented center $\bv^+(\theta_{S_{k_0}})$. Given $k\ge j$ and $\bu\in\cK_k(\theta_{S_k})$ in the same interval, define $\widetilde{\bu}$ on $R_j$ by
\[
    \widetilde u_i=
    \begin{cases}
        u_i,&i\in R_k,\\
        v_i^{(j)}(\theta_{S_j}),&i\in R_j\setminus R_k.
    \end{cases}
\]
For $j=0$, use $\bv^+(\theta_{S_{k_0}})$ in place of $\bv^{(j)}(\theta_{S_j})$ throughout this construction and the distance identity below. The newly pinned coordinates are filled with the stored center, not with their spin values $\pm1$. This keeps $\widetilde{\bu}$ in the analytic domain and gives the exact distance identity
\[
    \|\widetilde{\bu}-\bv^{(j)}(\theta_{S_j})\|
        =\|\bu-\bv^{(j)}(\theta_{S_j})|_{R_k}\|.
\]
At a positive checkpoint, $\widetilde{\bu}\in\cK_j(\theta_{S_j})$. For the initial interval, $\widetilde{\bu}\in D_{[n],n}\cap B(\bv^+(\theta_{S_{k_0}}),\rho_0\sqrt n)$, which is contained in the convexity region established in Step~1. In \eqref{eq:checkpoint-hessian-transfer}, the case $j=0$ uses $\Phi_0$ without a prefix argument. Since the entropy is a sum of one-coordinate terms, its $R_k\times R_k$ Hessian block is unaffected by the filled coordinates. We obtain the deterministic identity
\begin{equation}\label{eq:checkpoint-hessian-transfer}
    \nabla_{\bu}^2\Phi_k(\bu;\theta_{S_k})
       =\bigl(\nabla_{\bu}^2\Phi_j(\widetilde{\bu};\theta_{S_j})\bigr)_{R_k,R_k}
          +(a_{k/n}-a_{j/n})\bI.
\end{equation}
The revealed field contributes only a linear term and therefore does not enter this identity. A lower bound on a symmetric matrix passes to every principal submatrix, as can be seen by extending a test vector by zeros. Combining \eqref{eq:checkpoint-hessian-transfer} with \eqref{eq:scalar-checkpoint-lipschitz}, the only possible loss of convexity is the scalar change in $a$:
\[
    \nabla_{\bu}^2\Phi_j(\widetilde{\bu};\theta_{S_j})\succeq c\bI
    \quad\Longrightarrow\quad
    \nabla_{\bu}^2\Phi_k(\bu;\theta_{S_k})
       \succeq\left(c-L_\lambda\frac{k-j}{n}\right)\bI.
\]
There is no comparison of the entropy Hessian at two nearby arguments, so its divergence near $\pm1$ causes no difficulty. On the initial interval this gives
$\nabla_{\bu}^2\Phi_k(\bu;\theta_{S_k})\succeq(c_0/2-L_\lambda\alpha_0)\bI\succeq(c_0/4)\bI$ for every $\bu\in\cK_k(\theta_{S_k})$. On every later checkpoint interval, \eqref{eq:checkpoint-center-margins} gives
\[
    \nabla_{\bu}^2\Phi_k(\bu;\theta_{S_k})
       \succeq\{4\kappa_1-L_\lambda(\Delta+1/n)\}\bI
       \succeq2\kappa_1\bI,\qquad \bu\in\cK_k(\theta_{S_k}),
\]
for all sufficiently large $n$. The $1/n$ term accounts for rounding the checkpoints. These bounds hold simultaneously for all intermediate $k$ on the finite collection of checkpoint events; no union bound over the revealing times is needed.

\textbf{Step 4: Deal with $j_M \le k \le n$ and define the good events.} The Gaussian norm bound recorded in Section~\ref{ss:prelim-lem} handles the final suffix. At $j_M$, 
\[
    \|\bJ_{R_{j_M},R_{j_M}}\|_{\mathrm{op}}
       \le\lambda^2\frac{n-j_M}{n}
          +10\lambda\sqrt{\frac{n-j_M}{n}}<\frac12
\]
with probability tending to one. All subsequent free blocks are principal submatrices of this one. Since the entropy Hessian is at least $\bI$ and $a_{k/n}\ge0$, it follows simultaneously for $k\ge j_M$ that $\nabla_{\bu}^2\Phi_k(\bu;x_{S_k})\succeq\bI/2$ throughout $(-1,1)^{R_k}$ for every input prefix $x_{S_k}$. The same Gaussian norm bound in dimension $n$ gives $\|\bJ\|_{\mathrm{op}}\le\lambda^2+10\lambda$ with probability tending to one. It remains to check that clipping preserves the true means and to make the good events prefix-measurable. On the latter norm event, every row of $\bJ$ has absolute row sum at most $(\lambda^2+10\lambda)\sqrt n$. The exact single-site conditional field, with all other spins fixed, is $\sum_{j\ne i}J_{ij}\theta_j$; the diagonal term is constant in the Hamiltonian. Averaging its hyperbolic tangent over any unrevealed coordinates shows, for every prefix and every free coordinate,
\[
    |\operatorname{arctanh}(m_i^{(k)}(x_{S_k}))|
       \le(\lambda^2+10\lambda)\sqrt n,
    \qquad
    1-|m_i^{(k)}(x_{S_k})|\ge e^{-2(\lambda^2+10\lambda)\sqrt n}
       >2^{-n}
\]
for all sufficiently large $n$, since $\lambda$ is fixed and $n\log2>2(\lambda^2+10\lambda)\sqrt n$ eventually. Hence clipping does not exclude the true conditional mean. Together with the preceding ball-inclusion estimates this proves $\bm^{(k)}(\theta_{S_k})\in\cK_k(\theta_{S_k})$ simultaneously for every $k_0\le k<n$. For an arbitrary input prefix $x_{S_k}$ without a previous failure, each stored center belongs to its clipped cube, as does its restriction; this also holds for the initial orientation, even if it is zero. The center therefore belongs to the cube--ball intersection defining $\cK_k(x_{S_k})$. The whole-cube case contains zero, as do the domains assigned after a failure. Every domain is consequently nonempty, compact, and convex, independently of any probabilistic event.

With $\kappa_\lambda$ chosen in Step~2, define
\[
    G_k=\left\{\|\bJ\|_{\mathrm{op}}\le\lambda^2+10\lambda,
       \ \bm^{(k)}(\theta_{S_k})\in\cK_k(\theta_{S_k}),
       \ \inf_{\bu\in\cK_k(\theta_{S_k})}\lambda_{\min}(\nabla_{\bu}^2\Phi_k(\bu;\theta_{S_k}))
                   \ge\kappa_\lambda\right\}.
\]
On $G_k$, the Hessian bound and convexity of $\cK_k(\theta_{S_k})$ imply, for all $\bu,\bw\in\cK_k(\theta_{S_k})$,
\begin{equation}\label{eq:constraint-strong-convexity}
    \Phi_k(\bw;\theta_{S_k})\ge\Phi_k(\bu;\theta_{S_k})
       +\langle\nabla_{\bu}\Phi_k(\bu;\theta_{S_k}),\bw-\bu\rangle
       +\frac{\kappa_\lambda}{2}\|\bw-\bu\|^2.
\end{equation}
This follows by restricting $\Phi_k(\cdot;\theta_{S_k})$ to the segment joining $\bu$ and $\bw$. Compactness then gives a unique constrained minimizer. 
The conditional mean, objective, and stored center defining $\cK_k(\theta_{S_k})$ depend only on $\bJ$ and the already exposed index--value pairs. Thus $G_k\in\cH_k$; no unexposed index or value is used in its definition. The events employed above imply $\cap_{k=k_0}^{n-1}G_k$ with probability tending to one, so $b_n=o(1)$. This last conclusion uses maximal inequalities within the intervals and intersections over only finitely many checkpoint events.
\end{proof}

\subsection{Optimization accuracy under the prescribed tolerances}\label{ss:deterministic-tap-optimization}

In this subsection we prove Lemma~\ref{lem:optimization-accuracy}. Fix an input prefix $x_{S_k}$ and assume $\|\bJ\|_F\le n^2$. All objective values, gradients, projections, and outputs in this subsection use this same prefix. The local iterates $\bu^t$ and gradient approximations $\bg^t$ refer to this fixed call. Since $\operatorname{arctanh}(c_n)=\frac12\log(2^{n+1}-1)\le n+1$, the entropy gradient obeys $\|\operatorname{arctanh}(\bu)\|\le(n+1)\sqrt n$ on the clipped cube. Also
\[
    \|\bJ_{R_k,R_k}\bu\|\le\|\bJ\|_F\sqrt n,
    \qquad
    \|\widehat{\bh}^{(k)}(x_{S_k})\|\le\|\bJ\|_F\sqrt n,
    \qquad a_{k/n}\le\lambda^2.
\]
It follows that $\|\nabla_{\bu}\Phi_k(\bu;x_{S_k})\|\le B_n$ on $D_{R_k,n}$. The conservative choice of $B_n$ in Section~\ref{ss:conditional-tap}, and hence the prescribed step size and iteration count, remain valid with this sharper entropy-gradient bound. Every $\cK_k(x_{S_k})$ has diameter at most $D_n$. On the norm-capped input set, $B_n=O_\lambda(n^{5/2})$.

We use the deterministic implementation of Appendix~\ref{s:numerical-implementation}. The optimization proof requires the following gradient-error and feasible-projection bounds:
\[
    \|\bg^t-\nabla_{\bu}\Phi_k(\bu^t;x_{S_k})\|\le n^{-100},\qquad
    \bu^{t+1}\in\cK_k(x_{S_k}),\qquad
    \left\|\bu^{t+1}-\operatorname{Proj}_{\cK_k(x_{S_k})}
          (\bu^t-\gamma_n\bg^t)\right\|\le n^{-100}.
\]
Errors in the step size are included in the last tolerance. The coefficient and observation errors are included in the gradient tolerance. The returned vector is the feasible average of the stored pre-update iterates. The final logit-vector calculation has Euclidean error at most $n^{-100}$ relative to \eqref{eq:sampling-proposal-logits}, including the errors in the field, interaction matrix, and scalar coefficient. Appendix~\ref{s:numerical-implementation} realizes these bounds with polynomial bit cost.

Suppose that $\Phi_k(\cdot;x_{S_k})$ is convex on $\cK_k(x_{S_k})$, and let $\bu_*(x_{S_k})$ minimize it there. Apply the projected subgradient theorem \cite[Theorem~3.2]{bubeck2015convex} with gradient bound $B_n$, initial distance at most $D_n$, and $L_n$ iterations. The exact-arithmetic average satisfies
\[
    \Phi_k(\widehat{\bm}^{(k)}(x_{S_k});x_{S_k}) -\Phi_k(\bu_*(x_{S_k});x_{S_k}) \le\frac{D_n^2}{2\gamma_nL_n}+\frac{\gamma_nB_n^2}{2} =\frac{D_nB_n}{\sqrt{L_n}}=O_\lambda(n^{-17}).
\]
We only need to account for the prescribed gradient and projection errors, rather than repeat this standard convergence argument. For the rounded routine, nonexpansiveness of projection and feasibility give the same telescoping estimate with explicit additional terms:
\[
    \Phi_k(\widehat{\bm}^{(k)}(x_{S_k});x_{S_k})-\Phi_k(\bu_*(x_{S_k});x_{S_k}) \le\frac{D_n^2}{2\gamma_nL_n} +\frac{\gamma_n(B_n+n^{-100})^2}{2} +D_n n^{-100} +\frac{2D_n n^{-100}+n^{-200}}{2\gamma_n}.
\]
The term $D_n n^{-100}$ comes from the gradient error; the last term comes from the distance between the feasible stored iterate and the exact projection. Since $\gamma_n^{-1}=O_\lambda(n^{22})$, the additional terms are $O_\lambda(n^{-70})$. The theorem uses only convexity and a bounded gradient, not a polynomial uniform upper bound on the entropy Hessian; indeed, $\phi^{\prime\prime}(c_n)=(1-c_n^2)^{-1}\asymp2^n$. This implies \eqref{eq:sampling-optimization-accuracy}. 

\subsection{A uniform overlap estimate after logarithmic pinning}\label{ss:uniform-overlap}

We now prove the estimate stated in Proposition~\ref{prop:logarithmic-pinning}, together with the one-site cavity version needed for the TAP residual. The issue occurs at the beginning of the revealing path, where the side-channel strength is $t=\lambda^2k/n$. For the parameters $N=n-k-d$, $r=\lambda^2N/n$, and $q=F(rq+t)$, the identification in Section~\ref{ss:prelim-lem} gives
\[
    1-\tau=\frac{t}{rq+t}=\frac{k}{Nq+k}.
\]
For the symmetric Rademacher prior above the recovery threshold, the signed-overlap estimate in \cite[Proposition~13 and Theorem~7, equation~(26)]{alaoui2017finite} contains an exponential error $e^{-c(\tau)N}$ whose stated rate vanishes quadratically as $\tau\uparrow1$. At $k=O(\log n)$, $q\to q_0>0$ and $N(1-\tau)^2=O((\log n)^2/n)$. This degeneration prevents a direct uniform application at a logarithmic prefix.

We use Lemma~\ref{lem:unpinned-overlap-input} to control the magnitude of the overlap and an exact pinning comparison to control its sign. The cavity system in Lemma~\ref{lem:cavity-moment-input} then converts the resulting tail bound at a fixed tolerance into an $O(N^{-1})$ second moment. Its remainder estimate is used before signed-overlap concentration and remains uniform as $t\downarrow0$. At positive revealed fractions, Lemma~\ref{lem:overlap-concentration} supplies the required tail directly. Uniformity in the revealing index below refers to the averaged moments stated in the lemma.

\begin{lem}[Uniform overlap after logarithmic pinning]\label{lem:uniform-pinned-overlap}
For $d\in\{0,1\}$ and $k_0\le k<n-d$, set $N=n-k-d$, $r=\lambda^2N/n$, and $t=\lambda^2k/n$. In the Gaussian model of Section~\ref{ss:prelim-lem}, let $\sigma^1,\sigma^2$ be independent posterior replicas and put $R_{12}=N^{-1}\sum_{i=1}^N\sigma_i^1\sigma_i^2$. Then, uniformly in these choices,
\begin{equation}\label{eq:tap-overlap-second-moment}
    \bbE\left\langle(R_{12}-q)^2\right\rangle \le\frac{C_\lambda}{N},\qquad q=F(rq+t).
\end{equation}
If $\bm=\langle\sigma\rangle$ and $\bC=\langle(\sigma-\bm)(\sigma-\bm)^{\top}\rangle$ denote the posterior mean and covariance, then
\begin{equation}\label{eq:tap-cavity-covariance-moments}
    \bbE\left[\theta^{\top}\bC\theta+\bm^{\top}\bC\bm +\operatorname{Tr}(\bC^2)+(\|\bm\|^2-Nq)^2\right] \le C_\lambda N.
\end{equation}
Here brackets denote posterior averaging, including all replicas displayed inside them.
\end{lem}
\begin{proof}
Fix $d\in\{0,1\}$ and $k_0\le k<n-d$, and set $N=n-k-d$. We first use Lemma~\ref{lem:cavity-moment-input} to reduce the second-moment estimate to a concentration estimate at a fixed tolerance. We then obtain that concentration uniformly in $k$ by comparing the model to an unpinned posterior when $k$ is small, and by using Lemma~\ref{lem:overlap-concentration} when a positive fraction has been revealed.

The parameters $r=\lambda^2N/n$ and $t=\lambda^2k/n$ correspond to total signal strength $\lambda\sqrt{(n-d)/n}$ and revealed fraction $k/(n-d)$. The scalar bounds in Section~\ref{ss:prelim-lem}, applied uniformly in a compact neighborhood of the fixed value $\lambda>1$, therefore give, for all sufficiently large $n$,
\[
    q\ge c_\lambda>0,\qquad r(1-q)\le1-c_\lambda,\qquad |q-q_{k/n}|\le\frac{C_\lambda d}{n}.
\]
For $d=0$ the last difference is zero. For $d=1$, both the total signal strength and the revealed fraction differ from their original values by $O_\lambda(n^{-1})$, so the bound follows from the same implicit-function estimate. In particular,
\[
    \lambda^2(1-d/n)\le\Lambda=r+t/q \le\lambda^2+\lambda^2/c_\lambda.
\]
Thus $\Lambda$ stays in a fixed compact subinterval of $(1,\infty)$, and Lemma~\ref{lem:cavity-moment-input} applies uniformly throughout the present range. We next check that its linear system can be inverted with a uniform bound. Since $p\le q$, both eigenvalues listed in Lemma~\ref{lem:cavity-moment-input} are at most $r(1-q)\le1-c_\lambda$. Hence
\[
    \det(\bI-\bC_{\rm cav}) =\{1-r(1-2q+p)\}^2\{1-r(1-3q+2p)\} \ge c_\lambda^3.
\]
The entries of $\bC_{\rm cav}$ are bounded in terms of $\lambda$, so the adjugate formula gives $\|(\bI-\bC_{\rm cav})^{-1}\|_{\mathrm{op}}\le C_\lambda$. This argument also applies when the cavity matrix has a nontrivial Jordan block; it does not require symmetry. Inverting \eqref{eq:eakj-cavity-system}, using \eqref{eq:eakj-cavity-remainder}, and bounding $|R_{12}-q|\le2$, we obtain
\begin{equation}\label{eq:tap-cavity-moment-inequality}
    \bbE\left\langle(R_{12}-q)^2\right\rangle \le\frac{C_\lambda}{N} +C_\lambda\bbE\left\langle|R_{12}-q|^3\right\rangle.
\end{equation}
In particular, the constant is uniform even when $t$ tends to zero. Fix $\varepsilon\in(0,q_0/4)$ so small that the coefficient of the cubic moment in \eqref{eq:tap-cavity-moment-inequality}, multiplied by $\varepsilon$, is at most $1/2$. We will prove a uniform tail bound at this fixed tolerance. In fact, we will show that
\begin{equation}\label{eq:uniform-pinned-overlap-tail}
    \bbE\left\langle \bone_{\{|R_{12}-q|>\varepsilon\}}\right\rangle \le C_{\lambda,\varepsilon} \{n^{-16}+e^{-c_{\lambda,\varepsilon}N}\}.
\end{equation}
To see \eqref{eq:uniform-pinned-overlap-tail}, choose a constant $\delta>0$ sufficiently small, as specified below, and consider two cases separately.

\textbf{Case 1: $k_0\le k\le\delta n$.} To represent this model by exact pinning, generate an unpinned model on $n-d=N+k$ sites with off-diagonal interactions
\[
    A_{uv}=\frac{\lambda^2}{n}\theta_u\theta_v
                 +\frac{\lambda}{\sqrt n}Z_{uv},\qquad u<v,
\]
where the $Z_{uv}$ are independent standard Gaussians. Let $P$ be its posterior and let $S$ be an independent uniform $k$-subset of these sites. After revealing $\theta_S$, the field at a free site $j$ is
\[
    \sum_{a\in S}A_{ja}\theta_a
       =\frac{\lambda^2k}{n}\theta_j
           +\frac{\lambda}{\sqrt n}\sum_{a\in S}Z_{ja}\theta_a
       =t\theta_j+\sqrt t\,g_j.
\]
Conditional on $(\theta,S)$, the $g_j$ are independent standard Gaussians and are independent of the noise in the free interaction block. The latter block has signal coefficient $r/N=\lambda^2/n$ and noise variance $r/N$. After conditioning on $(\bA,S,\theta_S)$, the Hamiltonian terms supported entirely on $S$ are constant, so the free Gibbs density depends on the observations only through the free block and the displayed field. Consequently the posterior on the free sites has exactly the Gaussian side-channel law in the statement, after relabeling its coordinates. This also verifies the one-site deletion case $d=1$ without changing the interaction normalization. Let $\cF=\sigma(\bA,S)$. Under the planted law, the conditional distribution of $\theta$ given $\cF$ is $P$, by Lemma~\ref{lem:nishimori-input} and the independence of $S$. For fixed $\cF$, compare two ways of drawing a pair: the independent pair with law $P^{\otimes2}$, and the pair obtained by drawing $\theta\sim P$ followed by $\sigma\sim P(\cdot\mid\sigma_S=\theta_S)$. Writing $P(\theta_S)$ for the prefix mass $P\{x:x_S=\theta_S\}$, their exact density relation is
\begin{equation}\label{eq:pinning-density-ratio}
    P(\mathrm d\theta)P(\mathrm d\sigma\mid\sigma_S=\theta_S)
       =\frac{\bone_{\{\sigma_S=\theta_S\}}}{P(\theta_S)}
             P(\mathrm d\theta)P(\mathrm d\sigma).
\end{equation}
Conditional on $\sigma(\bA,S,\theta_S)$ under the second law, the free parts of $\theta$ and $\sigma$ are independent samples from the same pinned posterior. By Lemma~\ref{lem:nishimori-input}, their free overlap therefore has the same averaged law as $R_{12}$ in the $N$-site model. The estimate to be proved can thus be computed using the density in \eqref{eq:pinning-density-ratio}.

We first control its denominator. For every fixed $P$ and $S$,
\[
    P\{\theta:P(\theta_S)<2^{-k}n^{-16}\}
       =\sum_{x_S\in\{-1,1\}^S}
             P(x_S)\bone_{\{P(x_S)<2^{-k}n^{-16}\}}
       \le 2^k\,2^{-k}n^{-16}=n^{-16}.
\]
The marginal law of $\theta$ is still $P$ after pinning, so this exceptional event costs at most $n^{-16}$ under the second law as well. On its complement the reciprocal prefix mass is at most $2^kn^{16}$.
For a pair under $P^{\otimes2}$ put $\widetilde R=(n-d)^{-1}\langle\theta,\sigma\rangle$. The unpinned signal parameter is $\lambda^2(n-d)/n$, and its positive fixed point differs from $q_0$ by at most $C_\lambda d/n$. Lemma~\ref{lem:unpinned-overlap-input}, with its compact-parameter uniformity, therefore gives
\begin{equation}\label{eq:unpinned-amplitude-tail}
    \bbE P^{\otimes2}
       \left(\left||\widetilde R|-q_0\right|>\varepsilon/4\right)
       \le C_{\lambda,\varepsilon}e^{-c_{\lambda,\varepsilon}n}
\end{equation}
for all sufficiently large $n$, simultaneously for $d=0,1$.
On the agreement event $\{\sigma_S=\theta_S\}$, the full and free overlaps satisfy the exact relation
\[
    \widetilde R=\frac{k}{n-d}
       +\frac{N}{n-d}\left(\frac1N
                    \sum_{j\notin S}\sigma_j\theta_j\right).
\]
Their difference is therefore at most $2k/(n-d)$. Also,
$|q-q_0|\le C_\lambda(k+d)/n$ by the scalar estimates. Choose $\delta$ small enough that the sum of these two errors is at most $\varepsilon/2$ whenever $k\le\delta n$ and $n$ is sufficiently large. If
$\left||\widetilde R|-q_0\right|\le\varepsilon/4$ and $\widetilde R>-q_0/2$, then $\varepsilon<q_0/4$ forces $\widetilde R>0$. The free overlap is consequently within $3\varepsilon/4$ of $q$. Thus a free overlap outside $(q-\varepsilon,q+\varepsilon)$ implies either the amplitude event in \eqref{eq:unpinned-amplitude-tail} or the negative-phase event $\{\widetilde R\le-q_0/2\}$.

The random set $S$ suppresses the latter event. For a fixed pair $(\theta,\sigma)$, let $a$ be its number of agreeing coordinates, so $a/(n-d)=(1+\widetilde R)/2$. Sampling without replacement gives
\[
    \bbP_S(\sigma_S=\theta_S)
       =\frac{\binom ak}{\binom{n-d}k}
       \le\left(\frac{a}{n-d}\right)^k,
\]
where the ratio is zero if $a<k$. Hence, on $\widetilde R\le-q_0/2$, this probability is at most $\{(1-q_0/2)/2\}^k$. Under the comparison law, $S$ is independent of $\bA$ and the pair, so we can average over $S$ at this point. Combining the rare-prefix estimate, \eqref{eq:pinning-density-ratio}, the event inclusion above, and \eqref{eq:unpinned-amplitude-tail}, we obtain
\[
\begin{split}
    \bbE\left\langle
        \bone_{\{|R_{12}-q|>\varepsilon\}}\right\rangle
       \le{}&n^{-16}
          +C_{\lambda,\varepsilon}n^{16}2^k
                   e^{-c_{\lambda,\varepsilon}n}
          +n^{16}(1-q_0/2)^k.
\end{split}
\]
Decrease $\delta$ further so that $\delta\log2<c_{\lambda,\varepsilon}/4$; for large $n$, the middle term is then at most $C_{\lambda,\varepsilon}e^{-c_{\lambda,\varepsilon}n/2}$. For the last term, $q_0>1-\lambda^{-2}$ and the definition of $A_\lambda$ in \eqref{eq:sampling-prefix-length} imply
\[
    n^{16}(1-q_0/2)^k
       \le n^{16}\left(\frac{1+\lambda^{-2}}2\right)^{A_\lambda\log n}
       \le n^{-48},\qquad k\ge k_0.
\]

\textbf{Case 2: $k\ge\delta n$.} In this case we have $0\le r\le\lambda^2$ and $\lambda^2\delta\le t\le\lambda^2$. Lemma~\ref{lem:overlap-concentration}, applied on this compact parameter set, gives
\[
    \bbE\left\langle
        \bone_{\{|R_{12}-q|>\varepsilon\}}\right\rangle
       \le C_{\lambda,\varepsilon,\delta}
                    e^{-c_{\lambda,\varepsilon,\delta}N}.
\]
This includes arbitrarily small $r$; the case $r=0$ in the input supplies the endpoint needed for uniformity. Since $\delta$ has now been fixed in terms of $\lambda$ and $\varepsilon$, combining the two cases yields \eqref{eq:uniform-pinned-overlap-tail}.

Finally, split the absolute cubic moment according to $|R_{12}-q|\le\varepsilon$. Since $|R_{12}-q|\le2$, \eqref{eq:uniform-pinned-overlap-tail} gives
\[
\begin{split}
    \bbE\langle|R_{12}-q|^3\rangle
       &\le\varepsilon\bbE\langle(R_{12}-q)^2\rangle
            +8\bbE\langle\bone_{\{|R_{12}-q|>\varepsilon\}}\rangle\\
       &\le\varepsilon\bbE\langle(R_{12}-q)^2\rangle
            +C_{\lambda,\varepsilon}
                   \{n^{-16}+e^{-c_{\lambda,\varepsilon}N}\}.
\end{split}
\]
Substitute this into \eqref{eq:tap-cavity-moment-inequality}. By the choice of $\varepsilon$, the resulting multiple of the second moment can be moved to the left-hand side. Since $N\le n$, we have $n^{-16}\le N^{-1}$, and $e^{-cN}\le C_c/N$ for every fixed $c>0$. This proves \eqref{eq:tap-overlap-second-moment}. All choices were made in the order $\lambda$, the cavity constant, $\varepsilon$, $\delta$, and finally the lower bound on $n$. The finitely many smaller $n$ are absorbed into $C_\lambda$ using $|R_{12}-q|\le2$ and $N\le n$; if $k_0=n$, the assertion is vacuous. For the covariance assertion, return to the $N$-site posterior and condition on $(\bA,\bh)$. Independence of the two replicas gives
\[
    \left\langle\left(\sum_{i=1}^N \sigma_i^1\sigma_i^2\right)^2 \right\rangle =\operatorname{Tr}\bigl((\bC+\bm\bm^{\top})^2\bigr),
    \qquad \left\langle\sum_{i=1}^N\sigma_i^1 \sigma_i^2\right\rangle=\|\bm\|^2.
\]
Expanding the square therefore yields
\[
    N^2\left\langle(R_{12}-q)^2\right\rangle =\operatorname{Tr}(\bC^2)+2\bm^{\top}\bC\bm +(\|\bm\|^2-Nq)^2.
\]
Lemma~\ref{lem:nishimori-input} also gives
$\bbE[\theta\theta^{\top}\mid\bA,\bh]=\bC+\bm\bm^{\top}$, and hence
\[
    \bbE\theta^{\top}\bC\theta =\bbE\{\operatorname{Tr}(\bC^2)+\bm^{\top}\bC\bm\}.
\]
Because $\bC$ is positive semidefinite, all terms are nonnegative. The left-hand side of \eqref{eq:tap-cavity-covariance-moments} is thus at most twice $N^2\bbE\langle(R_{12}-q)^2\rangle$, which is at most $C_\lambda N$ by the estimate just proved.
\end{proof}

\subsection{The TAP gradient residual}\label{ss:tap-residual}

The goal of this subsection is to show the following estimate, which will be used in the proof of Lemma~\ref{lem:construct-scalar}.
\begin{lem}\label{lem:tap-gradient}
For $\theta\sim\mu_\bJ$ and the independent revealing order $I$, the target-prefix objective satisfies
\[
    \bbE\|\nabla_{\bu}\Phi_k(\bm^{(k)}(\theta_{S_k});\theta_{S_k})\|^2
       \le C_\lambda\frac{n-k}{n},\qquad k_0\le k<n.
\]
\end{lem}

We will prove Lemma~\ref{lem:tap-gradient} via the following cavity calculation. Fix $k$ and a free site $i\in R_k$. Delete site $i$ and all its incident interactions, retaining the fields $h_j^{(k)}(\theta_{S_k})$ for $j\in R_k\setminus\{i\}$. Denote the resulting Gibbs average by $\langle\cdot\rangle_{-i}$, its mean by $\bm^{-i}$, and its covariance by $\bC^{-i}$. In this part of the proof put $N=n-k-1$, initially assuming $N\ge1$. This is the model in Lemma~\ref{lem:uniform-pinned-overlap} with $d=1$; the normalization of each interaction remains $\lambda^2/n$, not $\lambda^2/N$. Let $\cF_{-i}=\sigma(\theta,I,(J_{ab})_{a\le b,\ a,b\ne i})$. The cavity mean and covariance are measurable with respect to $\cF_{-i}$. Conditional on $\cF_{-i}$, the removed row has the independent representation
\[
    J_{ij}=\frac{\lambda^2}{n}\theta_i\theta_j +\frac{\lambda}{\sqrt n}g_j, \qquad j\in R_k\setminus\{i\}, \qquad g_j\overset{\mathrm{i.i.d.}}\sim\cN(0,1).
\]
Define, for $-1\le s\le1$,
\begin{align*}
    Y_i(\sigma)&=\sum_{j\in R_k\setminus\{i\}} J_{ij}(\sigma_j-m_j^{-i}),\qquad v_i=\frac{\lambda^2}{n}(N-\|\bm^{-i}\|^2),\\
    \cZ_i(s)&=\langle e^{sY_i(\sigma)}\rangle_{-i},\qquad D_i(s)=\frac{\cZ_i'(s)}{\cZ_i(s)}-s v_i.
\end{align*}
Since $\langle Y_i\rangle_{-i}=0$, Jensen's inequality gives $\cZ_i(s)\ge1$. We will first show the following estimate.
\begin{lem}{\label{lem:bound-D-i}}
    We have
    \begin{equation}\label{eq:tap-centered-cavity-field}
        \sup_{|s|\le1}\bbE|D_i(s)|^2 \le\frac{C_\lambda N}{n^2}.
    \end{equation}
\end{lem}
\begin{proof}
Here is the Gaussian calculation, with all conditioning taken before the removed row is observed. For two independent cavity replicas $\sigma,\tau$, set
\begin{align*}
    d(\sigma)&=\frac{\lambda^2}{n}\theta_i \sum_{j\in R_k\setminus\{i\}}\theta_j(\sigma_j-m_j^{-i}),\\
    e(\sigma)&=\frac{\lambda^2}{n}\|\sigma-\bm^{-i}\|^2-v_i =-\frac{2\lambda^2}{n} \langle\bm^{-i},\sigma-\bm^{-i}\rangle,\\
    c(\sigma,\tau)&=\frac{\lambda^2}{n} \langle\sigma-\bm^{-i},\tau-\bm^{-i}\rangle.
\end{align*}
These variables are bounded in absolute value by $C_\lambda$, and each has zero cavity average. For auxiliary real parameters $s,t\in[-1,1]$, put
\[
    U_{s,t}=s d(\sigma)+t d(\tau) +\frac{s^2}{2}e(\sigma)+\frac{t^2}{2}e(\tau) +st c(\sigma,\tau).
\]
Note that
\[
    s Y_i(\sigma) + t Y_i(\tau) = sd(\sigma) +td(\tau) + \frac{\lambda}{\sqrt n} \sum_{ j \in R_k \setminus \{ i \} } g_j \big[ s(\sigma_j-m_j^{-i}) + t(\tau_j-m_j^{-i}) \big].
\]
Writing $\bbE_{\bg}$ for expectation over the removed standard Gaussian row $(g_j)_{j\in R_k\setminus\{i\}}$ only, with $\cF_{-i}$ held fixed, its moment generating function gives the exact identity. Equivalently, this averages the original row $(W_{ij})_{j\in R_k\setminus\{i\}}$, with $g_j=\sqrt n\,W_{ij}$:
\begin{align*}
    \bbE_{\bg}[\cZ_i(s)\cZ_i(t)]
       &=\left\langle \bbE_{\bg}e^{sY_i(\sigma)+tY_i(\tau)}
                    \right\rangle_{-i}^{\otimes2}\\
       &=e^{v_i(s^2+t^2)/2}
                    \langle e^{U_{s,t}}\rangle_{-i}^{\otimes2}.
\end{align*}
Apply $(\partial_s-v_i s)(\partial_t-v_i t)$ and then set $t=s$. Abbreviating $c=c(\sigma,\tau)$ inside the next display, we obtain
\[
    \bbE_{\bg}|\cZ_i'(s)-s v_i\cZ_i(s)|^2
    =e^{v_i s^2}\left\langle
       \big\{[d(\sigma)+s e(\sigma)+sc]
              [d(\tau)+s e(\tau)+sc]+c\big\}
                  e^{U_{s,s}}\right\rangle_{-i}^{\otimes2}.
\]
The product term is bounded by a constant times the sum of the second moments of $d,e,c$. For the remaining term, use $\langle c\rangle_{-i}^{\otimes2}=0$ to replace $e^{U_{s,s}}$ by $e^{U_{s,s}}-1$, and then use $|e^u-1|\le C_\lambda|u|$ for $|u|\le C_\lambda$. Therefore, since $\cZ_i(s)\ge1$,
\[
\begin{split}
    \bbE_{\bg}|D_i(s)|^2
    &\le C_\lambda\left\langle
        d(\sigma)^2+e(\sigma)^2+c(\sigma,\tau)^2
                      \right\rangle_{-i}^{\otimes2}\\
    &\le\frac{C_\lambda}{n^2}\left\{
       \theta_{R_k\setminus\{i\}}^\top\bC^{-i}
                  \theta_{R_k\setminus\{i\}}
       +(\bm^{-i})^\top\bC^{-i}\bm^{-i}
       +\operatorname{Tr}((\bC^{-i})^2)\right\}.
\end{split}
\]
Averaging over the cavity observations and the planted vector and applying Lemma~\ref{lem:uniform-pinned-overlap}, specifically \eqref{eq:tap-cavity-covariance-moments}, proves \eqref{eq:tap-centered-cavity-field}.
\end{proof}

\begin{proof}[Proof of Lemma~\ref{lem:tap-gradient}]
Using $\widehat{\bh}^{(k)}(\theta_{S_k})=\bh^{(k)}(\theta_{S_k})$, $\boldsymbol{\ell}^{(k)}(\theta_{S_k}) = \operatorname{arctanh}(\bm^{(k)}(\theta_{S_k}))$ and differentiating in the optimization variable gives
\[
    \nabla_{\bu}\Phi_k(\bm^{(k)}(\theta_{S_k});\theta_{S_k}) =\boldsymbol{\ell}^{(k)}(\theta_{S_k}) -\bJ_{R_k,R_k}\bm^{(k)}(\theta_{S_k})-\bh^{(k)}(\theta_{S_k}) +a_{k/n}\bm^{(k)}(\theta_{S_k}).
\]
We now relate the centered cavity field to the full marginal and identify the Onsager correction. Let $p_i=(1+m_i^{(k)}(\theta_{S_k}))/2$ and $b_i=\sum_{j\in R_k\setminus\{i\}}J_{ij}m_j^{-i}$. Conditioning the full posterior on the sign at $i$ gives the exact identities
\[
\begin{split}
    \ell_i^{(k)}(\theta_{S_k})
       &=h_i^{(k)}(\theta_{S_k})+b_i
          +\frac12\{\log\cZ_i(1)-\log\cZ_i(-1)\},\\
    \sum_{j\in R_k\setminus\{i\}}J_{ij}m_j^{(k)}(\theta_{S_k})
       &=b_i+p_i\frac{\cZ_i'(1)}{\cZ_i(1)}
          +(1-p_i)\frac{\cZ_i'(-1)}{\cZ_i(-1)}.
\end{split}
\]
Since $2p_i-1=m_i^{(k)}(\theta_{S_k})$ and
\[
    \log\cZ_i(1) - \log\cZ_i(-1) = \int_{-1}^{1} \frac{ \cZ_i'(s) }{ \cZ_i(s) } \de s = \int_{-1}^{1} \big( D_i(s)+ v_i s \big) \de s = \int_{-1}^{1} D_i(s)\,\de s,
\]
substitution into the gradient gives
\[
    (\nabla_{\bu}\Phi_k(\bm^{(k)}(\theta_{S_k});\theta_{S_k}))_i =\frac12\int_{-1}^1D_i(s) \de s -p_iD_i(1)-(1-p_i)D_i(-1) +(a_{k/n}-v_i-J_{ii})m_i^{(k)}(\theta_{S_k}).
\]
For the deterministic-coefficient comparison, let $q$ be the cavity fixed point used above. Since $|q-q_{k/n}|\le C_\lambda/n$,
\[
    a_{k/n}-v_i =\frac{\lambda^2}{n}\left\{ \|\bm^{-i}\|^2-Nq+(1-q_{k/n}) +N(q-q_{k/n})\right\}.
\]
Lemma~\ref{lem:uniform-pinned-overlap}, through \eqref{eq:tap-cavity-covariance-moments}, consequently gives
$\bbE|a_{k/n}-v_i|^2\le C_\lambda(N+1)/n^2$. Since $|m_i^{(k)}(\theta_{S_k})|\le1$,
\[
    \bbE|J_{ii}m_i^{(k)}(\theta_{S_k})|^2 \le\bbE J_{ii}^2 =\frac{2\lambda^2}{n}+\frac{\lambda^4}{n^2}.
\]
Using $0\le p_i\le1$, Cauchy--Schwarz for the integral, and \eqref{eq:tap-centered-cavity-field}, we conclude that
\[
    \bbE\big|(\nabla_{\bu}\Phi_k(\bm^{(k)}(\theta_{S_k});\theta_{S_k}))_i\big|^2 \le C_\lambda\left\{\frac{N+1}{n^2}+\frac1n\right\} \le\frac{C_\lambda}{n}.
\]
If $N=0$, the cavity sums are empty, $\ell_i^{(k)}(\theta_{S_k})=h_i^{(k)}(\theta_{S_k})$, and the residual is simply $(a_{k/n}-J_{ii})m_i^{(k)}(\theta_{S_k})$; the same bound follows from $a_{k/n}\le\lambda^2/n$ and the displayed diagonal estimate. Exchangeability allows the exposed set and the free site to be fixed when computing each annealed expectation. Summing over the $n-k$ free sites and then averaging over the independent revealing order proves
\[
    \bbE\|\nabla_{\bu}\Phi_k(\bm^{(k)}(\theta_{S_k});\theta_{S_k})\|^2 \le C_\lambda\frac{n-k}{n}.
\]
The finitely many smaller dimensions are absorbed by enlarging $C_\lambda$; their residuals have finite second moments because the exact marginal logit is bounded by the absolute row sum of $\bJ$. 
\end{proof}

\subsection{From the TAP residual to the proposal logits} \label{ss:sampling-guarantee}

We now prove Lemma~\ref{lem:construct-scalar} using Lemmas~\ref{lem:convexity-barycenter}, \ref{lem:optimization-accuracy}, and~\ref{lem:tap-gradient}. We compare $\widehat{\boldsymbol\ell}^{(k)}(\theta_{S_k})$ with $\boldsymbol\ell^{(k)}(\theta_{S_k})$, both at the same posterior-distributed prefix. In particular, the input field is $\widehat{\bh}^{(k)}(\theta_{S_k})=\bh^{(k)}(\theta_{S_k})$, and the optimization is over $\cK_k(\theta_{S_k})$. The gradient of the objective defined in Section~\ref{ss:conditional-tap} then gives the exact identity
\begin{equation}{\label{eq:gradient-identity}}
    \begin{aligned}
    \boldsymbol\ell^{(k)}(\theta_{S_k})
       &=\operatorname{arctanh}(\bm^{(k)}(\theta_{S_k}))\\
       &=\bh^{(k)}(\theta_{S_k})
          +(\bJ_{R_k,R_k}-a_{k/n}\bI)\bm^{(k)}(\theta_{S_k})
          +\nabla_{\bu}\Phi_k(\bm^{(k)}(\theta_{S_k});\theta_{S_k}).
    \end{aligned}
\end{equation}

\begin{proof}[Proof of Lemma~\ref{lem:construct-scalar}]
We use the domain conclusion of Lemma~\ref{lem:convexity-barycenter} and the separate TAP-residual estimate in Lemma~\ref{lem:tap-gradient}. Enlarge the exceptional events, if necessary, to include failures of the initial computation and all safeguards encountered so far. This preserves prefix measurability and $b_n=o(1)$ for a fixed valid record. The numerical stability and fallback estimates needed here are verified in Appendix~\ref{ss:tuning-free-sampler}, using the precision rules of Appendix~\ref{ss:precision-implementation}. Fix $k$ and a realization of the target prefix $\theta_{S_k}$ for which $G_k$ holds, and let $\bu_*(\theta_{S_k})$ minimize $\Phi_k(\cdot;\theta_{S_k})$ over $\cK_k(\theta_{S_k})$. The minimizer exists because $\cK_k(\theta_{S_k})$ is nonempty and compact, and is unique by strong convexity. Since $\bm^{(k)}(\theta_{S_k})\in\cK_k(\theta_{S_k})$, the variational inequality at $\bu_*(\theta_{S_k})$ gives
\[
\begin{split}
    \kappa_\lambda\|\bm^{(k)}(\theta_{S_k})-\bu_*(\theta_{S_k})\|^2
    &\le\left\langle
       \nabla_{\bu}\Phi_k(\bm^{(k)}(\theta_{S_k});\theta_{S_k})-\nabla_{\bu}\Phi_k(\bu_*(\theta_{S_k});\theta_{S_k}),
       \bm^{(k)}(\theta_{S_k})-\bu_*(\theta_{S_k})\right\rangle\\
    &\le\left\langle\nabla_{\bu}\Phi_k(\bm^{(k)}(\theta_{S_k});\theta_{S_k}),
       \bm^{(k)}(\theta_{S_k})-\bu_*(\theta_{S_k})\right\rangle.
\end{split}
\]
Hence $\|\bm^{(k)}(\theta_{S_k})-\bu_*(\theta_{S_k})\|\le\kappa_\lambda^{-1}\|\nabla_{\bu}\Phi_k(\bm^{(k)}(\theta_{S_k});\theta_{S_k})\|$. The optimization guarantee \eqref{eq:sampling-optimization-accuracy} and \eqref{eq:constraint-strong-convexity}, together with the variational inequality at $\bu_*(\theta_{S_k})$, also yield
\[
    \frac{\kappa_\lambda}{2}\|\widehat{\bm}^{(k)}(\theta_{S_k})-\bu_*(\theta_{S_k})\|^2
       \le\Phi_k(\widehat{\bm}^{(k)}(\theta_{S_k});\theta_{S_k})-\Phi_k(\bu_*(\theta_{S_k});\theta_{S_k})
       \le n^{-12}.
\]
Note that \eqref{eq:gradient-identity} and \eqref{eq:sampling-proposal-logits} use the same input $\theta_{S_k}$, so their field terms cancel exactly as $\widehat{\bh}^{(k)}(\theta_{S_k})=\bh^{(k)}(\theta_{S_k})$. Including the final numerical error gives
\[
\begin{aligned}
    \widehat{\boldsymbol\ell}^{(k)}(\theta_{S_k})
        -\boldsymbol\ell^{(k)}(\theta_{S_k})
       &=(\bJ_{R_k,R_k}-a_{k/n}\bI)
          \bigl(\widehat{\bm}^{(k)}(\theta_{S_k})-\bm^{(k)}(\theta_{S_k})\bigr)\\
       &\qquad-\nabla_{\bu}\Phi_k(\bm^{(k)}(\theta_{S_k});\theta_{S_k})
          +\be^{(k)}(\theta_{S_k}),
\end{aligned}
\]
where $\|\be^{(k)}(\theta_{S_k})\|\le n^{-100}$.
On $G_k$, the operator norm of $\bJ_{R_k,R_k}-a_{k/n}\bI$ is bounded by a constant depending only on $\lambda$. Combining the preceding estimates, and absorbing the prescribed numerical errors, gives
\[
    \big\|\widehat{\boldsymbol\ell}^{(k)}(\theta_{S_k})-\boldsymbol\ell^{(k)}(\theta_{S_k})\big\|^2
       \le C_\lambda\left\{
          \|\nabla_{\bu}\Phi_k(\bm^{(k)}(\theta_{S_k});\theta_{S_k})\|^2+n^{-12}\right\}.
\]
Taking expectations and using the TAP-residual bound proves the first inequality in \eqref{eq:vector-logit-error}, since $n-k\ge1$.

The exact conditional mean $\bm^{(k)}(\theta_{S_k})$ is $\cH_k$-measurable: conditioning additionally on the unused portion of the independent permutation does not change it. The computed vector $\widehat{\boldsymbol\ell}^{(k)}(\theta_{S_k})$ and $G_k$ are also $\cH_k$-measurable, while $I_{k+1}$ is conditionally uniform on $R_k$. Therefore
\[
    \bbE\left[\bone_{G_k}
       \big|\widehat\ell^{(k)}_{I_{k+1}}(\theta_{S_k})-\ell^{(k)}_{I_{k+1}}(\theta_{S_k})\big|^2
       \,\middle|\,\cH_k\right]
       =\frac{\bone_{G_k}}{n-k}
          \big\|\widehat{\boldsymbol\ell}^{(k)}(\theta_{S_k})-\boldsymbol\ell^{(k)}(\theta_{S_k})\big\|^2.
\]
This proves the second inequality in \eqref{eq:vector-logit-error}. 
Moreover, summing the first inequality after division by $n-k$ gives
\[
    \bbE\sum_{k=k_0}^{n-1}\frac{\bone_{G_k}}{n-k}
       \big\|\widehat{\boldsymbol\ell}^{(k)}(\theta_{S_k})-\boldsymbol\ell^{(k)}(\theta_{S_k})\big\|^2
       \le C_\lambda.
\]
On $\cap_{k=k_0}^{n-1}G_k$ all the indicators equal one. Markov's inequality and the probability $b_n$ of its complement prove the asserted pathwise, averaged posterior-mass estimate.

Finally, every conditional factor in \eqref{eq:def-nu} is positive. To evaluate $Q(x)$, replay exactly the same deterministic computations using the supplied prefix $x_{S_k}$, including its phase orientation and checkpoint centers. This requires at most $n$ polynomial-time prefix computations; no enumeration of other prefixes or access to previous proposals is needed. The capped routines specify the proposal on every input, and their failures are included in the exceptional events above. Thus both sampling and density evaluation take polynomial time.
\end{proof}

\section{Completion of the proof}\label{s:completion}

We now combine the proposal guarantee with the Metropolis correction. For a fixed valid rational record, Lemma~\ref{lem:nu-is-warm} gives the required warm proposal directly. To obtain the algorithm with input only $(\bY,\lambda)$, Appendix~\ref{ss:tuning-free-sampler} constructs $n$ candidate proposals, one of which, say $Q_{a_*}$, comes from a fixed valid record for all sufficiently large $n$. Their average satisfies $Q(x)\ge Q_{a_*}(x)/n$ for every $x$. Thus the warmness threshold increases by only a factor $n$, and no accuracy guarantee for the other candidates is needed. We retain both versions in the algorithm below; the binary encoding and numerical details are deferred to the appendices.

\begin{proof}[Proof of Theorem~\ref{thm:sampling}]
Fix $\lambda>1$. We describe both the fixed-record and the tuning-free versions of the sampler. In the former, use a fixed valid rational record obtained from Lemma~\ref{lem:convexity-barycenter} as explained in Appendix~\ref{ss:tuning-free-sampler}, and put $d=1$. In the latter, put $d=n$ and use the first $n$ capped records in Appendix~\ref{ss:tuning-free-sampler}, assigning the uniform proposal to every record that fails its checks. For each retained record, fix the deterministic numerical rules and compute its initial stored vector before drawing the common independent order $I$. Conditional on $(\bJ,I)$, let $Q_1,\ldots,Q_d$ denote the resulting component proposals, including their prescribed suffix fallbacks, and set
\[
    Q(s)=\frac1d\sum_{a=1}^d Q_a(s),\qquad s\in\{-1,1\}^n.
\]
All records, initial stored vectors, and the order are kept fixed throughout the Metropolis run.

In Algorithm~\ref{alg:posterior-sampler}, the subroutine \textsc{DrawAndScore} first chooses $a$ uniformly from $[d]$ and generates $s$ from $Q_a$ by Algorithm~\ref{alg:sequential-tap} in \textsc{Sample} mode. It then evaluates $Q_b(s)$ for every $b\in[d]$ by the same algorithm in \textsc{Evaluate} mode, and returns $(s,\ell_Q)$, where
\[
    \ell_Q=\log Q(s)
       =\operatorname{logsumexp}_{1\le b\le d}\{\log Q_b(s)\}-\log d,
    \qquad
    \operatorname{logsumexp}_b\{u_b\}=\log\sum_b e^{u_b}.
\]
The selected component's log-density may be reused. Every other evaluation starts its prefix computations afresh, using the supplied configuration. 

\medskip
\begin{breakablealgorithm}
\caption{Posterior sampling by sequential TAP proposals and Metropolis correction}
\label{alg:posterior-sampler}
\begin{algorithmic}[1]
\Require Observation $\bY$ and known $\lambda>1$; optionally, a fixed valid rational tuning record.
\Ensure A configuration $\widehat\theta\in\{-1,1\}^n$ in the original coordinate labels.
\If{$n=1$}
    \State \Return a uniform sign.
\EndIf
\State Set $\bJ\gets\lambda\bY/\sqrt n$.
\State Apply the input and norm safeguard of Appendix~\ref{ss:input-safeguard}; on failure return an independent uniform configuration.
\State Compute $k_0$ by \eqref{eq:sampling-prefix-length}; attempt the shared precomputation of $a_{k/n}$ for $k_0\le k<n$ and record any failure.
\State Use the prescribed record ($d=1$), or the first $n$ capped records ($d=n$), as in Appendix~\ref{ss:tuning-free-sampler}.
\State On shared-precomputation failure assign every component the uniform proposal; otherwise apply the record checks in Appendix~\ref{ss:tuning-free-sampler}.
\State Precompute and store the initial vector of each remaining component; on failure assign that component the uniform proposal.
\State Draw one independent uniform permutation $I$; keep it and all initial stored vectors fixed thereafter.
\State Set $Q\gets d^{-1}\sum_{a=1}^dQ_a$ and $T_n\gets\lceil2^{k_0}dn^2\rceil$.
\State $(s,\ell_Q)\gets\textsc{DrawAndScore}(\bJ,I,Q_1,\ldots,Q_d)$; set $H\gets H_{\bJ}(s)$.
\For{$t=1,\ldots,T_n$}
    \State Independently draw $(s',\ell'_Q)\gets\textsc{DrawAndScore}(\bJ,I,Q_1,\ldots,Q_d)$.
    \State $H'\gets H_{\bJ}(s')$ and $\Delta\gets H'-H+\ell_Q-\ell'_Q$.
    \State Draw $U\sim\operatorname{Unif}(0,1)$ independently.
    \If{$\log U\le\min\{0,\Delta\}$}
        \State $(s,H,\ell_Q)\gets(s',H',\ell'_Q)$.
    \EndIf
\EndFor
\State \Return $\widehat\theta=s$.
\end{algorithmic}
\end{breakablealgorithm}
\medskip

By the component construction and replay argument in Appendix~\ref{ss:tuning-free-sampler}, \textsc{DrawAndScore} produces independent $Q$-distributed proposals together with their mixture scores, conditional on the frozen data and order. The normalization and positivity conclusions of Lemma~\ref{lem:tuning-free-proposal} apply to every component, including those assigned a uniform suffix or the uniform law.

For the warmness estimate, apply Lemma~\ref{lem:tuning-free-proposal} with $u=\log n$. Since $\bJ=\lambda\bY/\sqrt n$ is an invertible deterministic transformation of the observation, $\mu_\bJ=\mu_\bY$. With
\[
    M=2^{k_0}dn,\qquad
    \delta=\mu_\bY\left(\frac{\mu_\bY}{Q}>M\right),
\]
we therefore have
\[
    \bbE_{\bY,I}\delta
       \le b_n+\frac{C_\lambda+1/e}{\log n}=o(1).
\]
Here $b_n=o(1)$ includes the exceptional initialization and numerical events for the one fixed valid component used in that lemma. If $k_0=n$, the mixture is uniform and $\delta=0$ directly.

We can now apply the Metropolis correction. Fix $(\bY,I)$ for which the input and norm safeguard does not trigger. With the proposal $Q$ frozen, the acceptance probability in Algorithm~\ref{alg:posterior-sampler} is
\[
    \min\left\{1,
       e^{H_{\bJ}(s')-H_{\bJ}(s)}\frac{Q(s)}{Q(s')}\right\}
       =\min\left\{1,
         \frac{\mu_{\bY}(s')Q(s)}{\mu_{\bY}(s)Q(s')}\right\}.
\]
The posterior normalizing constant cancels. Hence this is precisely Algorithm~\ref{alg:metropolis-correction}, initialized from $Q$, with
$T_n=\lceil Mn\rceil=\lceil2^{k_0}dn^2\rceil$.
Writing $K$ for its transition kernel, Lemma~\ref{lem:metro-correction} gives, whenever $\delta\le1/2$,
\[
    \|QK^{T_n}-\mu_{\bY}\|_{\TV}
       \le\delta+\frac\delta M
           +\exp\left\{-\frac{T_n(1-\delta)}M\right\}
       \le2\delta+e^{-n/2}.
\]
The last bound also holds for $\delta>1/2$, because total variation is at most one.

The input and norm safeguard of Appendix~\ref{ss:input-safeguard} rejects on an event of probability $O_\lambda(n^{-3})$ and otherwise guarantees $\|\bJ\|_F\le n^2$. On the rejection event the algorithm returns a uniform configuration, contributing at most its probability to the annealed total variation bound. Averaging the preceding Metropolis estimate, and then using convexity of total variation to average over $I$ conditional on $\bY$, proves, for the algorithm with exact random draws and scores,
\begin{equation}\label{eq:completed-sampler-tv}
    \bbE_{\bY}\|
       \widehat\mu(\cdot\mid\bY)-\mu_{\bY}\|_{\TV}
       \le 2b_n+\frac{2(C_\lambda+1/e)}{\log n}
          +e^{-n/2}+O_\lambda(n^{-3})=o(1).
\end{equation}
The case $n=1$ is exact because the diagonal observation carries no information about its uniform sign.

It remains to bound the running time and account for the finite-precision random draws.
For $n\ge2$, the definition of $k_0$ gives
\[
    2^{k_0}\le4n^{A_\lambda\log2},\qquad
    T_n\le4n^{A_\lambda\log2+3}+1,
\]
since $d\le n$. The first $n$ binary strings have length $O(\log n)$ and can be enumerated and checked in polynomial time. All retained iteration counts and the number of checkpoints are at most $n$. Their initial spectral and AMP computations, with the prescribed precision and caps, therefore take polynomial time. Each sequential generation or evaluation uses at most $n$ prefix computations, each with $L_n=\lceil n^{40}\rceil$ projected-gradient iterations in Algorithm~\ref{alg:tap-optimization}. A gradient evaluation and a Hamiltonian evaluation require at most $O(n^2)$ arithmetic operations, apart from the coordinatewise elementary functions. A call to \textsc{DrawAndScore} uses one generation and at most $d$ component evaluations. It never enumerates configurations or alternative prefixes.

Appendix~\ref{s:numerical-implementation} realizes these deterministic computations with polynomial bit cost. Its stepwise coupling changes the output law by at most $O(n^{-2}(T_n+1)^{-1})=o(1)$ when random draws and scores are implemented with finite precision. Adding this error to \eqref{eq:completed-sampler-tv} proves \eqref{eq:main-tv} for the implemented sampler. Markov's inequality then gives the stated convergence in probability. Both versions run in time $n^{O_\lambda(1)}$, completing the proof.
\end{proof}

\appendix

\section{A specification without unknown tuning constants}\label{ss:tuning-free-sampler}

The construction in Sections~\ref{s:algorithm}--\ref{s:auxiliary-proofs} gives a proposal once a valid tuning record has been fixed. Its validity depends on constants in the local-convexity and AMP estimates, whose numerical values need not be known. We now specify a single algorithm with input only $(\bY,\lambda)$. It forms a finite family of proposals and uses their average. The role of the preceding analysis is to guarantee that this family eventually contains one suitable component; the algorithm only checks the descriptions and executes the prescribed computations.

We first explain why it suffices to consider rational records. Choose the integer $L_{\rm init}$ supplied by Lemma~\ref{lem:unpinned-tap-input}. Its initialization guarantee holds for every step size in a fixed open interval, so we may choose a rational $\gamma$ in that interval at the outset. In Step~1 of the proof of Lemma~\ref{lem:convexity-barycenter}, choose rational $\rho_0,\alpha_0>0$ satisfying the displayed strict inequalities. In Step~2, choose the terminal fraction $\alpha_M$ rational, still satisfying the strict bound that makes the final interaction block small. After the finite-cover argument has supplied positive auxiliary radii and curvature constants, choose $\rho$ rational in its permitted interval and take an integer $M$ sufficiently large. The equally spaced mesh
\[
    \alpha_\ell=\alpha_0+
           \frac\ell M(\alpha_M-\alpha_0),\qquad 0\le\ell\le M,
\]
then has rational entries and satisfies the required mesh inequalities. Finally, choose the fixed integer counts $L_0,\ldots,L_{M-1}$ as in Step~3. This produces a fixed valid rational record. The auxiliary constants used to prove its validity are not entries of the record. All these choices are made for a fixed $\lambda$ before taking $n\to\infty$; in particular, we do not approximate an already chosen step size inside an iteration whose length grows with $n$.

For definiteness, fix the following binary encoding, independently of $\lambda$ and $n$. If $h$ is a positive integer and its ordinary binary expansion $\operatorname{bin}(h)$ has length $b$, encode $h$ by the word $1^b0\,\operatorname{bin}(h)$. The initial run of ones specifies the length of the final block, so concatenated integer codes can be parsed uniquely. Encode each positive rational as the pair of codes for its numerator and denominator in lowest terms. Encode a record in the order
\[
    M,\ L_{\rm init},\ L_0,\ldots,L_{M-1},\
    \gamma,\rho_0,\rho,\ \alpha_0,\ldots,\alpha_M.
\]
The first integer determines the required number of subsequent entries. The decoder rejects incomplete codes, noncanonical integer or rational encodings, missing entries, and strings with unused trailing bits. Parsing, coprimality tests, and comparisons of the rational entries use integer arithmetic.

Enumerate all finite binary strings, starting with the empty string, in increasing length and then lexicographic order. At dimension $n$, retain the first $n$ positions. Their maximum length is $\lfloor\log_2 n\rfloor$, so enumeration and decoding require polynomial time. A fixed code of length $b$ occurs among the first $2^{b+1}-1$ positions. Malformed strings are assigned the uniform proposal on $\{-1,1\}^n$. Their positions are retained in the family, so its size is always $n$.

For every decoded record, check the positive-integer requirements, $0<\gamma<1$, $\rho_0,\rho>0$, and $0<\alpha_0<\cdots<\alpha_M<1$. Also require
\begin{equation}\label{eq:tuning-record-dimension-checks}
    M,L_{\rm init},L_0,\ldots,L_{M-1}\le n,
    \qquad n^{-1}\le\gamma,\rho_0,\rho\le n,
    \qquad k_0<j_0<\cdots<j_M<n,
\end{equation}
where $j_\ell=\lfloor\alpha_\ell n\rfloor$. Check $M\le n$ before allocating arrays indexed by the checkpoints, and check every iteration count before starting its loop. The rational inequalities and floors are computed exactly from the encoded integers. Any record failing a check is assigned the uniform proposal. These tests use only the record, $n$, and the known $k_0$; they involve neither a posterior mean nor a test of local convexity.

Fix one valid rational record obtained above, and let $a_*=a_*(\lambda)$ be the position of its code in this enumeration. The index $a_*$ is finite and independent of $n$. Its fixed integer entries are eventually at most $n$, and its fixed positive scalar entries eventually belong to $[n^{-1},n]$. Moreover,
\[
    \frac{k_0}{n}\longrightarrow0,\qquad
    \alpha_0>0,\qquad
    \min_{0\le\ell<M}(\alpha_{\ell+1}-\alpha_\ell)>0,
    \qquad \alpha_M<1.
\]
These facts imply the rounded-checkpoint inequalities in \eqref{eq:tuning-record-dimension-checks} for all sufficiently large $n$. Hence this record eventually occurs among the retained strings and passes every description and dimension check. Neither its index nor the dimension at which this happens is needed by the algorithm. If $k_0=n$, every record fails the rounded-checkpoint check and all components are uniform, which is also a valid proposal specification.

We next describe the computations and safeguards for each retained component. The common scalar coefficients are computed by the certified routine in Appendix~\ref{ss:scalar-computation}. For each record that passes its checks, compute the initial stored vector by Section~\ref{ss:conditional-tap} before drawing the common independent order $I$. A failure of the common coefficient computation assigns all components the uniform proposal. A failure in the initialization of one record assigns that component the uniform proposal. In either case the corresponding TAP calls are skipped. All component definitions, stored initial vectors, and numerical rules are then fixed, and the same order $I$ is used for the entire Metropolis run.


For the fixed valid record, the budgets and numerical errors are compatible with the probabilistic inputs. At fixed $\lambda$, the common scalar precomputation succeeds for all sufficiently large $n$ by the certified polynomial bounds in Appendix~\ref{ss:scalar-computation}. The initial eigensolver safeguard is controlled in Appendix~\ref{ss:precision-implementation}. On the operator-norm event of Lemma~\ref{lem:convexity-barycenter}, the natural-gradient update in the $\bz$ coordinates has a Lipschitz constant bounded by $C_\lambda$, as follows from its displayed recursion and the bounded, $1$-Lipschitz function $\tanh$. The prescribed spectral and update errors therefore accumulate over the $n$ refinement steps to at most
\[
    n^{O(1)}C_\lambda^n2^{-n^4}=o(1).
\]
The fixed initial and checkpoint AMP counts admit the corresponding finite-iteration estimate. Together with the clipping error, these perturbations are $o(\sqrt n)$ and are absorbed by the strict ball-inclusion and curvature margins in Section~\ref{ss:checkpoint-construction}, in particular \eqref{eq:checkpoint-center-margins}. The constrained optimization and final logit errors are controlled by Lemma~\ref{lem:optimization-accuracy} and the tolerances in Section~\ref{ss:deterministic-tap-optimization}. Thus the exceptional initialization and numerical events for this one record may be included in its $b_n=o(1)$. We require this conclusion only for the fixed valid record.

For a component that reaches the sequential stage, apply Algorithm~\ref{alg:sequential-tap}. If a computation fails at the prefix $x_{S_k}$, set the entire current logit vector to zero and use conditional probabilities $1/2$ for the current coordinate and the remaining suffix. Retain the already exposed prefix. The computation and its failure tests are completed before inspecting $I_{k+1}$ or its supplied sign. If the field update fails after the current sign has been generated, the uniform suffix begins at the following step, as in the algorithm. All later TAP computations in that call are skipped.

These rules define a normalized, strictly positive mass function $Q_a$ for every component. Indeed, every finite logit defines two positive conditional probabilities adding to one, and a fallback uses the pair $(1/2,1/2)$. If the first fallback occurs after $k$ exposed coordinates, then every continuation $x$ of that prefix has probability
\[
    Q_a(x)=Q_a(x_{S_k})\,2^{-(n-k)},
\]
where $Q_a(x_{S_k})$ is the probability of reaching the prefix under the same sequential rule. Summing over its $2^{n-k}$ continuations preserves this prefix mass. Equivalently, summing the product of conditional factors successively over the coordinates gives total mass one. A component rejected before generation has $Q_a(x)=2^{-n}$ for every $x$.

To evaluate a component on a supplied configuration, start its sequential state and operation counters afresh and read that configuration in the fixed order. Recompute its phase orientation, checkpoint centers, fields, and all failure decisions from the supplied prefixes. The initial stored vector is reused, while the prefix-dependent quantities are determined within this call. Replay consequently reproduces the same deterministic logits and conditional factors that define $Q_a$. The log-density is evaluated to the prescribed accuracy using the stable formulas in Appendix~\ref{ss:random-implementation}. Increasing the accuracy of a score calculation leaves the deterministic optimization tolerances, and hence the underlying proposal, fixed. We distinguish these ideal conditional factors from their finite-precision random draws and scores; their approximation error is accounted for in Appendix~\ref{ss:random-implementation}.

For the version with no supplied record, put $d=n$ and let $Q_1,\ldots,Q_d$ be the components just constructed. If a fixed valid rational record is supplied, use it alone, put $d=1$, and apply the same numerical and fallback rules. In either case define
\begin{equation}\label{eq:tuning-mixture-definition}
    Q(x)=\frac1d\sum_{a=1}^dQ_a(x),
       \qquad x\in\{-1,1\}^n.
\end{equation}
For each proposal draw, choose a fresh uniform $a\in[d]$ and generate from $Q_a$. Conditional on the frozen data and order, this gives distribution $Q$. To obtain its score, evaluate every component on the same configuration and use
\[
    \log Q(x)=\log\left(\sum_{a=1}^d
                              e^{\log Q_a(x)}\right)-\log d.
\]
The generating component's score can be reused. Uniform components and components that use a suffix fallback are included in this sum. Both scores in the Metropolis acceptance ratio are scores of $Q$; the randomly selected index is used only to generate the proposal. The complete procedure is given in Algorithm~\ref{alg:posterior-sampler}.

The following consequence of the fixed-record guarantee explains why this construction removes the unknown tuning constants.

\begin{lem}[Proposal without a supplied tuning record]\label{lem:tuning-free-proposal}
Fix $\lambda>1$ and form $Q$ by \eqref{eq:tuning-mixture-definition}, with $d=n$ for the enumerated construction or $d=1$ for a supplied fixed valid rational record. There are $C_\lambda<\infty$ and $b_n=o(1)$ such that, for all sufficiently large $n$ and every $u>0$,
\begin{equation}\label{eq:tuning-mixture-warmness}
    \bbE_{\bJ,I}\mu_\bJ\left(
       \frac{\mu_\bJ}{Q}>d\,2^{k_0}e^u\right)
       \le b_n+\frac{C_\lambda+1/e}{u}.
\end{equation}
Every component is normalized and strictly positive. The deterministic computations for generating a proposal and evaluating its log-density have polynomial cost for fixed $\lambda$, with random draws and score approximations implemented as in Appendix~\ref{ss:random-implementation}.
\end{lem}

\begin{proof}
Normalization, positivity, replay, and the uniform polynomial bounds on the retained computations were established above. In the enumerated construction, choose the fixed valid rational record at position $a_*$, which is present and passes its description checks for all sufficiently large $n$. For a supplied fixed valid rational record take $a_*=1$. Lemmas~\ref{lem:construct-scalar} and~\ref{lem:nu-is-warm}, with the exceptional events described above included in $b_n$, apply to $Q_{a_*}$.

For every realization of the data and order, all summands in \eqref{eq:tuning-mixture-definition} are nonnegative, so $Q(x)\ge Q_{a_*}(x)/d$ for every $x$. Consequently,
\[
    \left\{x:\frac{\mu_\bJ(x)}{Q(x)}>d\,2^{k_0}e^u\right\}
       \subseteq
    \left\{x:\frac{\mu_\bJ(x)}{Q_{a_*}(x)}>2^{k_0}e^u\right\}.
\]
Taking posterior masses, averaging over $(\bJ,I)$, and applying \eqref{eq:proposal-warmness-tail} proves \eqref{eq:tuning-mixture-warmness}. This argument analyzes only the fixed component $a_*$; it requires no accuracy estimate or simultaneous good event for the other components. When $k_0=n$, all components are uniform and $\mu_\bJ/Q\le2^n$, so the tail event is empty.
\end{proof}

Taking $u=\log n$ in Lemma~\ref{lem:tuning-free-proposal} gives a known warmness threshold $2^{k_0}dn$ outside a set of expected posterior mass $o(1)$. Lemma~\ref{lem:metro-correction} therefore leads to the deterministic correction time
\[
    T_n=\left\lceil2^{k_0}dn^2\right\rceil.
\]
For the enumerated construction, the factor $d=n$ increases the warmness threshold by $n$ and gives $T_n=\lceil2^{k_0}n^3\rceil$. Evaluating all $n$ component scores contributes a further factor $n$ to the cost of a proposal call. Both costs are polynomial, since $2^{k_0}=n^{O_\lambda(1)}$. The algorithm uses neither the index of the valid record nor the constants governing its exceptional probability; Section~\ref{s:completion} applies the correction bound, and Appendix~\ref{s:numerical-implementation} completes the finite-precision analysis.

\section{Numerical implementation}\label{s:numerical-implementation}

This appendix realizes the computational convention in Section~\ref{s:algorithm}. We first specify the input safeguard and deterministic numerical maps. We then approximate the random draws and scores without changing those maps. All routines have polynomial bit cost for fixed $\lambda>1$, and the final coupling bounds the additional output-law error independently of the probabilistic sampling argument.

\subsection{Input and norm safeguard}\label{ss:input-safeguard}

For $n\ge2$, the input routine computes a symmetric dyadic approximation $\bJ^{\rm dyad}$ with entrywise error at most $2^{-n^4}$, subject to the caps in Appendix~\ref{ss:precision-implementation}. Failure returns an independent uniform configuration. Otherwise, using exact rational arithmetic, reject if
\[
    \sum_{i,j}(J^{\rm dyad}_{ij})^2>(n^2-1)^2.
\]
This is the safeguard called by Algorithm~\ref{alg:posterior-sampler}.

The global input and norm safeguard has vanishing probability. Its squared-norm comparison is exact rational arithmetic, so it requires no decision at a real-valued threshold. After successful input acquisition, $n\ge2$ and $\|\bJ-\bJ^{\rm dyad}\|_F\le n2^{-n^4}<1$, so passing the safeguard guarantees $\|\bJ\|_F\le n^2$, while rejection implies $\|\bJ\|_F>n^2-2$. For fixed $\lambda$, on $\|\bJ\|_F\le n^2-2$ the observation entries are bounded by $C_\lambda n^{5/2}$, and only $O_\lambda(n^4+\log n)$ bits per entry are needed for the prescribed approximation. Thus input acquisition and the norm calculation succeed within their caps on this event for all sufficiently large $n$. Their combined failure event is consequently contained in $\{\|\bJ\|_F>n^2-2\}$. Indeed, under the planted law,
\[
\begin{aligned}
    \bbP(\text{input or norm safeguard rejects})
       &\le\bbP(\|\bJ\|_F>n^2-2)\\
       &\le\frac{\lambda^4+\lambda^2(n+1)}{(n^2-2)^2}
         =O_\lambda(n^{-3}).
\end{aligned}
\]

\subsection{Scalar computation and prefix rounding}\label{ss:scalar-computation}

Use the positive scalar solution and the certified derivative lower bound from Section~\ref{ss:prelim-lem}. For example, truncate the standard Gaussian integral at $|G|\le C\sqrt{\log n}$ and use a deterministic quadrature with polynomially small mesh and certified error. On the bounded range $0\le r_\alpha q+t_\alpha\le\lambda^2$, the required integrand derivatives are uniformly bounded for fixed $\lambda$. The Gaussian tail and quadrature error can both be made $O(n^{-130})$. Bisection requires $O_\lambda(\log n)$ steps. If a certified interval for the fixed-point residual contains zero, the certified derivative lower bound on the bracket either already certifies the requested root accuracy or permits refinement of the quadrature. Thus this computation never requires an exact sign test at a root.

For the prefix length in \eqref{eq:sampling-prefix-length}, use the following deterministic rule. In the finite-precision implementation, take the ceiling of a certified upper approximation to $A_\lambda\log n$ with absolute error at most $1/4$, and then cap it at $n$. This may increase $k_0$ by at most one, never decreases it, and avoids having to decide equality at an integer boundary. All subsequent bounds hold with this implemented value of $k_0$ as well.

\subsection{Precision rules and capped routines}\label{ss:precision-implementation}

The deterministic numerical routines use the precision tolerances in Section~\ref{ss:deterministic-tap-optimization}, the bound $2^{n^3}$ on stored intermediate magnitudes, and caps of $n^{1000}$ working bits and $n^{1000}$ bit operations for each invocation of an input-acquisition, norm-evaluation, scalar, spectral, AMP, or projection routine (including its elementary-function evaluations). Fix standard polynomial-time integer arithmetic, elementary-function, and symmetric-eigenvalue implementations before running the algorithm. These deliberately generous caps dominate the costs described in Section~\ref{ss:deterministic-tap-optimization} and Appendices~\ref{ss:scalar-computation}, \ref{ss:projection-implementation}, and~\ref{ss:random-implementation} for a fixed valid record and all sufficiently large $n$. For example, the scalar quadrature may use a mesh of width at most $n^{-140}$ on its truncated Gaussian interval, followed by the certified bisection in Appendix~\ref{ss:scalar-computation}. Each routine either returns its prescribed certified output within its budget or returns a failure flag. Each invocation starts with a fresh operation counter; the rule therefore gives the same outcome when replayed on the same input. The checks in \eqref{eq:tuning-record-dimension-checks}, the fixed $L_n=\lceil n^{40}\rceil$ optimization steps, and these numerical budgets ensure polynomial running time for every retained component, including those without a convexity guarantee.

Fix the following precision rules independently of the prefix and of the next unrevealed index. Use a symmetric dyadic approximation of $\bJ$ with entrywise error at most $2^{-n^4}$, and approximate the known parameter to the same accuracy, increasing precision when necessary to meet the coefficient tolerance in Section~\ref{ss:conditional-tap} and the vector tolerances in Section~\ref{ss:deterministic-tap-optimization}. Compute the initial spectral vector, each initial or checkpoint AMP update, and each natural-gradient update to Euclidean error at most $2^{-n^4}$. Use a capped eigensolver that certifies a gap larger than $n^{-1}$ between the two largest eigenvalues; failure of this test invokes the initial uniform safeguard. For fixed $\lambda>1$, the separated leading eigenvalue used in the spectral initialization of \cite{montanari2021estimation,celentano2023local} makes this safeguard inactive with probability tending to one. On that event, matrix perturbation changes the normalized spectral vector up to a global sign by at most $n^{O(1)}2^{-n^4}$; an additional $O_\lambda(\log n)$ bits suffice to meet the stated vector tolerance. The initial AMP and natural-gradient recursions are odd, and the prefix orientation removes a common sign change. Thus this eigenspace perturbation estimate does not require continuity of the deterministic eigenvector sign convention.

Store each center as a dyadic vector and the projected iterates on a common dyadic grid containing that center. Take the exact rational average of the stored pre-update iterates; this average is feasible. Projection and feasible rounding are specified in Appendix~\ref{ss:projection-implementation}. Safeguard comparisons use certified intervals, with an unresolved comparison treated as a failure; the strict margins on the good events make these extra failures negligible. The requested precision is polynomial in $n$; in particular, evaluating $\operatorname{arctanh}$ on the clipped cube requires only $O(n+\log n)$ bits of working precision for the stated absolute accuracy, since $1-|u_i|\ge2^{-n}$.

\subsection{Projection and feasible rounding}\label{ss:projection-implementation}

We also make explicit that projection onto the prescribed domains is a polynomial-time operation. Fix an input prefix $x_{S_k}$. For its cube--ball domain $\cK_k(x_{S_k})$, write $\bv$ locally for the restricted stored center specified in \eqref{eq:sampling-checkpoint-domains}, let $\by$ be the vector to project, and consider
\[
    \bu(s)=\operatorname{Proj}_{D_{R_k,n}}
          \bigl(\bv+s(\by-\bv)\bigr),\qquad 0\le s\le1.
\]
Here cube projection is coordinatewise clipping. Since $\bv$ belongs to the cube, $\|\bu(s)-\bv\|$ is continuous and nondecreasing in $s$. If $\bu(1)$ lies in the ball, it is the required projection. Otherwise, the multiplier optimality equations show that the projection is $\bu(s_*)$, where $s_*$ is the largest $s$ for which $\bu(s)$ lies in the ball. Bisection on $[0,1]$, retaining the feasible endpoint, approximates this point while preserving feasibility. Feasibility is decided by comparing the squared distance with $\rho^2n$ (or $\rho_0^2n$), an exact rational comparison at dyadic bisection points; equality is accepted. The bound
$\|\bu(s)-\bu(t)\|\le|s-t|\,\|\by-\bv\|$
shows that polynomially many bisection steps suffice for the prescribed accuracy, even under the magnitude caps. After choosing the feasible bisection endpoint, round each coordinate toward the corresponding coordinate of the stored rational center. Each coordinate then stays in the cube and its distance to the center can only decrease, so the rounded point remains in the ball. A sufficiently fine common dyadic grid makes the total Euclidean rounding error smaller than the prescribed projection tolerance. In the whole-cube case, clipping followed by rounding toward zero is feasible. The scalar quadrature and bisection, the clipped entropy gradient, and this feasible rounding are performed to the accuracies in Sections~\ref{ss:conditional-tap} and~\ref{ss:deterministic-tap-optimization}. The dyadic endpoint $c_n=1-2^{-n}$ has an exact $O(n)$-bit representation; the initialization accuracy $2^{-n^4}$ likewise requires only polynomially many bits. Thus the deterministic computations defining every proposal have polynomial cost, including for records without any convexity guarantee.

\subsection{Scores, random draws, and output-law comparison}\label{ss:random-implementation}

Fix the deterministic numerical maps specified above first; their allowed errors have already been incorporated in Lemma~\ref{lem:construct-scalar}. They define the same positive component probabilities in generation and replay. It remains only to approximate the random draws and the log-density and acceptance calculations. Accumulate component log-densities using
\[
    \log(2\cosh u)=|u|+\log(1+e^{-2|u|}),
\]
and mixture log-densities using log-sum-exp. In particular, exponentially small probabilities do not require multiplication in ordinary floating-point arithmetic. Take the total proposal-draw error in total variation and the uniform acceptance-probability error per complete Metropolis step, as well as the error in the initial random choices, at most
\[
    \xi_n=n^{-2}(T_n+1)^{-2}.
\]
This can be achieved by approximating each sequential Bernoulli probability to absolute error $\xi_n/(10n)$ and the Hamiltonian and mixture log-scores to absolute error $\xi_n/10$, with smaller fixed constants if needed. Indeed, a sequential coupling sums the Bernoulli errors, log-sum-exp is $1$-Lipschitz in the maximum norm, and $u\mapsto\min\{1,e^u\}$ is $1$-Lipschitz. Rounding the final random choices to sufficiently fine dyadic probabilities also has arbitrarily small absolute error. The magnitude caps and $\log(1/\xi_n)=O_\lambda(\log n)$ make these operations polynomial-time for fixed $\lambda$.

The independent permutation and component choices also admit a bounded-cost implementation. To approximate a uniform integer in $[d]$, draw $b$ fair bits, regard them as $U\in\{0,\ldots,2^b-1\}/2^b$, and return $1+\lfloor dU\rfloor$. Its total variation error is at most $d2^{-b}$. Applying this to the successive choices in a permutation and to each mixture index, with $b$ large enough that their total error per relevant draw is at most $\xi_n$, uses $O_\lambda(\log n)$ bits per choice. This avoids an unbounded rejection loop. The same finite-bit construction realizes Bernoulli draws to the stated absolute accuracy. In the comparison below, the ideal order is exactly uniform and independent of the data; the approximation error of its implemented version is charged only once.

Couple the ideal and implemented chains step by step, using the same fixed proposal definition. Until the first discrepancy, their states coincide; the discrepancy probability at each step is $O(\xi_n)$. The additional output-law error is therefore at most $O((T_n+1)\xi_n)=o(1)$. Since $(T_n+1)\xi_n=n^{-2}(T_n+1)^{-1}$, this is the additional error used in Section~\ref{s:completion}. The construction uses polynomially many bit operations for each fixed $\lambda>1$ and for every retained record, completing the numerical implementation.

\subsection*{Acknowledgment}

This work is partially supported by the NSFC Major Program (Project No. 12595294), the NSFC Key Program (Project No. 12231002), and the New Cornerstone Investigator Program (Project No. NCI202501).

\subsection*{Statement on AI usage}

This project grew out of the author's exploration on applying Lemma~\ref{lem:metro-correction} in several related sampling tasks. During its development, the author interacted extensively with GPT-6 Astra, posing the problem to it and providing it with a set of notes from earlier work and from the related models. Working from these notes, GPT-6 Astra suggested the martingale argument in Lemma~\ref{lem:convexity-barycenter}, the tuning-constants free specification in Appendix~\ref{ss:tuning-free-sampler}, and concluded the proof following successive rounds of interactive refinement. In addition, the author used large language models to polish and edit the writing. The author verified all mathematical content themselves, wrote the introduction and the complete proof in their present form, and takes full responsibility for the content of this paper.

\bibliographystyle{alpha}
\bibliography{Z2sampling}

\end{document}